\documentclass[11pt,reqno]{amsart}
\usepackage[T1]{fontenc}
\usepackage{amsmath,amsthm,amssymb,mathrsfs}
\usepackage[abbrev,non-sorted-cites]{amsrefs}
\usepackage{hyperref}
\usepackage{color,graphicx}
\usepackage{lmodern,microtype,mathtools,bm,enumitem,needspace}
\usepackage{bookmark}
\usepackage{array,booktabs,tabularx}
\usepackage{placeins}
\usepackage[nameinlink,capitalize,noabbrev]{cleveref}
\microtypesetup{expansion=false}
\hypersetup{hypertexnames=false,hidelinks,
  pdfauthor={Wei-Chun Chen, Wei-Ting Kao, and Chun-Kai Lien},
  pdftitle={Stable Rigidity in the Heisenberg Groups},
  pdfsubject={Centre reconstruction and volume-preserving CMC stability},
  pdfkeywords={Heisenberg group, Pansu sphere, horizontal perimeter, Jacobi operator,
  stable CMC, centre reconstruction, Morse index, characteristic points, variation}}
\setlist[itemize]{leftmargin=2.2em,itemsep=0.15em,topsep=0.25em}
\numberwithin{equation}{section}

\newcommand{\PansuResultBookmark}[3]{%
  \ifcsname r@#2\endcsname%
    \bookmark[level=#1,dest={\getrefbykeydefault{#2}{anchor}{}}]{%
      #3 \ref*{#2}: \nameref*{#2}}%
  \fi%
}

\theoremstyle{plain}
\newtheorem{thm}{Theorem}[section]
\newtheorem{theorem}[thm]{Theorem}
\newtheorem{lemma}[thm]{Lemma}
\newtheorem{proposition}[thm]{Proposition}
\newtheorem{corollary}[thm]{Corollary}
\theoremstyle{definition}

\newtheorem*{AIuse}{AI Usage Statement}

\newtheorem*{applicationremark}{Remark}

\newcommand{\RR}{\mathbb R}
\newcommand{\CC}{\mathbb C}
\newcommand{\Hn}{\mathbb H^n}
\newcommand{\LTheta}{{L_\Theta}}
\newcommand{\gTheta}{{g_\Theta}}
\newcommand{\PSphere}{\mathcal S}
\newcommand{\Sphere}{S^{2n-1}}
\newcommand{\gSphere}{{g_{S^{2n-1}}}}
\newcommand{\TSphere}{T_S}
\newcommand{\CRD}{\xi_S}
\newcommand{\Shape}{\mathfrak S}
\newcommand{\AngleError}{\mathsf E_u}
\newcommand{\ShapeError}{\mathsf E_P}
\newcommand{\PH}{P_H}
\newcommand{\dd}{\,\mathrm d}
\newcommand{\abs}[1]{\lvert#1\rvert}
\newcommand{\ip}[2]{\langle #1,#2\rangle}
\DeclareMathOperator{\Id}{Id}
\let\Re\relax
\let\Im\relax
\DeclareMathOperator{\Re}{Re}
\DeclareMathOperator{\Im}{Im}
\allowdisplaybreaks[1]

\title[Stable Rigidity]{Stable Rigidity in the Heisenberg Groups}
\author{Wei-Chun Chen}
\address{Department of Mathematics, National Taiwan University, Taipei 10617, Taiwan}
\email{b14201028@ntu.edu.tw}

\author{Wei-Ting Kao}
\address{Institute of Mathematics, Academia Sinica, Taipei 106319, Taiwan}
\email{wtkao@as.edu.tw}

\author{Chun-Kai Lien}
\address{Department of Mathematics, National Taiwan University, Taipei 10617, Taiwan}
\email{d13221002@ntu.edu.tw}

\date{}

\newcommand{\PansuCMapPath}{}
\IfFileExists{\detokenize{$TEXMFOUTPUT}/\jobname-pansu-oms.cmap}
  {\renewcommand{\PansuCMapPath}{\detokenize{$TEXMFOUTPUT}/}}
  {\IfFileExists{\detokenize{$TEXMF_OUTPUT_DIRECTORY}/\jobname-pansu-oms.cmap}
    {\renewcommand{\PansuCMapPath}{\detokenize{$TEXMF_OUTPUT_DIRECTORY}/}}{}}
\AtBeginDocument{%
}

\begin{document}
\begin{abstract}
For $n\ge2$, we prove that the $C^2$ boundary of a bounded nonempty
open set in $\Hn$ is volume-preserving stable if and only if it is
a Pansu sphere up to left translation and Heisenberg dilation.
The proof combines centre reconstruction with a volume-constrained
test field derived from Pansu dilation.
Consequently, every isoperimetric region admitting an
open representative with $C^2$ boundary is a Pansu ball up to these
symmetries.
\end{abstract}
\maketitle

\section{Introduction}\label{sec:introduction}
We classify volume-preserving stable $C^2$ boundaries of bounded
open sets in the Heisenberg groups $\Hn$, $n\ge2$.
Stationarity supplies constant horizontal mean curvature (CMC).

The Euclidean prototype is the characterization of compact stable CMC
hypersurfaces by Barbosa and do Carmo~\cite{BarbosaDoCarmo}.
Wente's proof~\cite{Wente} uses normal displacement and dilation.
Here an exact Hessian identity controls the variation of a
reconstructed centre and the departure of the normal angle from
its Pansu relation. Stability enforces both conditions, yielding
the classification without assuming rotational symmetry or umbilicity.

This article is submitted concurrently with our companion paper,
\emph{The Jacobi Spectrum and Global Second Variation of Pansu Spheres}~\cite{JacobiSpectrum}.
That paper provides the second-variation framework in its
Section~2 and proves the stability of Pansu spheres under ambient
variations moving the characteristic poles in its Corollary~1.4.
The present paper proves the rigidity direction.

Pansu spheres are candidates
for the sub-Riemannian isoperimetric problem~\cites{Pansu,CDPT}.
Rotational CMC hypersurfaces were studied by Ritor\'e and
Rosales~\cite{RitoreRosales} and Chiu, Lai and Liu~\cite{CLL};
positive-CMC umbilic spheres were characterized by Cheng, Chiu,
Hwang and Yang~\cite{CCHY}*{Theorem~A and Corollary~A$'$}.
\pagebreak[3]
Their subsequent work~\cite{CCHYSigma}*{Theorem~B} establishes rotational
symmetry for closed connected umbilic hypersurfaces. Chiu, Lai and
Liu~\cite{CLL}*{Corollary~4.4} explicitly classify closed connected embedded
umbilic hypersurfaces of positive constant horizontal mean curvature
as Pansu spheres. The remark following the proof of
Proposition~\ref{prop:zero-defect-reconstruction} compares the
regularity assumptions in these results with those used here.

In $\mathbb H^1$, Ritor\'e and Rosales
\cite{RitoreRosalesStationary}*{Theorem~6.10} proved that every compact,
connected, $C^2$ immersed volume-preserving area-stationary surface
is a Pansu sphere. For stability results, see
\cite{HurtadoRosalesStable}*{Corollary~5.5} and
\cite{HurtadoRosalesSpheres}*{Theorem~5.9 and Appendix~B}.
For $n\ge2$, Montefalcone established radial and local stability of
Pansu spheres \cite{MontefalconeStability}*{Proposition~4.9 and Theorem~4.14}.
Hladky and Pauls~\cite{HPPreprint}*{Section~9 and Appendix~B}
proposed a proof of volume-preserving stability in $\mathbb H^2$
using spherical harmonics and a radial inequality. This argument
appears in their preprint but is absent from the published
article~\cite{HP}.

Write $\Sigma^*$ for the noncharacteristic part of a boundary $\Sigma$;
throughout, \emph{regular} means noncharacteristic.
We say that an open set has $C^2$ boundary if, near each boundary
point, it lies on one side of a $C^2$ graph and has that graph as
its boundary. For a bounded open set, this boundary is compact
and embedded. We write $\PH(\Sigma)$ for horizontal perimeter and
$V(\Omega)=|\Omega|$ for Haar volume, where $\Sigma=\partial\Omega$.

\Needspace{5\baselineskip}
We call a boundary \emph{volume-preserving stationary} if its first
perimeter derivative vanishes along every smooth ambient variation
preserving the enclosed volume exactly. It is
\emph{volume-preserving stable} if it is stationary and the second
perimeter derivative is nonnegative along every such variation.
Lemma~\ref{lem:stationarity-cmc} derives a unique multiplier $h$
and $H_\Sigma=h$ on $\Sigma^*$, and identifies stability with
nonnegativity of the Hessian of $\PH-hV$ on all $C^2$ ambient fields
satisfying the linearized volume constraint $\delta V[X]=0$.

\PansuResultBookmark{2}{rig:thm:stable-cmc}{Theorem}
\begin{theorem}[Stability rigidity]\label{rig:thm:stable-cmc}\label{cor:index-one}
Let $n\ge2$, and let $\Omega\subset\Hn$ be a bounded nonempty open set
with $C^2$ boundary $\Sigma$, oriented by its outward
unit normal. The following are equivalent:
\begin{enumerate}[label=\textup{(\roman*)},leftmargin=2.2em]
\item $\Sigma$ is a left translate and a Heisenberg dilation of a Pansu sphere;
\item $\Sigma$ is volume-preserving stable;
\item $\Sigma$ is volume-preserving stationary and
      $\operatorname{ind}(\mathcal B_h)=1$.
\end{enumerate}
In \textup{(iii)}, $h$ is the multiplier supplied by stationarity
and $\mathcal B_h$ is the ambient Hessian of $\PH-hV$; its index
is the supremum of dimensions of negative-definite spaces of $C^2$
ambient fields. Whenever these conditions hold, $H_\Sigma=h>0$
on $\Sigma^*$ and the sphere has radius $2n/h$.

\end{theorem}

After dilation to $H_\Sigma=2n$, write $\mathcal B=\mathcal B_{2n}$.
The model can be recognized by two conditions:
its reconstructed centre is constant, and its normal
angle $u$ has horizontal gradient $\nabla_bu=-e_n$, where $e_n$
is the characteristic direction. On a general boundary we consider the
centre differential $D_Hz_0$ and the departure
$\AngleError=\nabla_bu+e_n$ from the Pansu angle relation.
The normalized dilation speed of the model defines a function $F$; an ambient
realization $X_F$ and a multiple of the dilation field $D$ give the
volume-constrained test field $W=X_F-\kappa D$, which satisfies
$\delta V[W]=0$.
In the exact identity of Theorem~\ref{rig:prop:quantitative-instability},
the linearized volume constraint contributes a negative multiple of
the squared perimeter average of $\langle\AngleError,e_n\rangle$. Cauchy--Schwarz controls
this term and gives
\begin{equation}\label{eq:intro-centre-identity}
 -\mathcal B(W,W)
 \ge\int_\Sigma\left(|D_Hz_0|^2+\frac{4n+1}{2n+1}|\AngleError|^2\right)d\sigma_H.
\end{equation}
Section~\ref{subsec:geometric-characterization} constructs $z_0$
and proves the corresponding characterization. Although $W$ is only $C^1$, smooth
approximation preserving the linearized volume constraint extends stability to it. The
left side is then nonpositive, so $D_Hz_0=\AngleError=0$.
Centre reconstruction identifies every component as a unit Pansu sphere;
volume-preserving combinations of componentwise dilations rule out multiple
components when the boundary is stable. On a connected normalized
boundary, nonnegativity on $W$ alone already suffices
(Corollary~\ref{thm:single-field-classification}).

Dilation orthogonality gives the index equivalence.
In both rigidity and model stability, normalized speeds may be
singular even for smooth ambient fields. Differentiation
and weak divergence extension justify the global identities at
$C^2$ regularity.

\Needspace{8\baselineskip}
The classification has the following consequence for the isoperimetric problem.
\PansuResultBookmark{2}{rig:cor:isoperimetric}{Corollary}
\begin{corollary}[Isoperimetric classification under boundary regularity assumptions]
\label{rig:cor:isoperimetric}
Let $E\subset\Hn$, $n\ge2$, be an isoperimetric region
with $0<|E|<\infty$. Suppose $E$ agrees up to a null set with an
open set $\Omega$ having $C^2$ boundary.
Then, up to a null set, $E$ is a left translate and a Heisenberg
dilation of a Pansu ball.
\end{corollary}

Section~\ref{sec:regular} develops the centre reconstruction and the
geometric characterization of Pansu spheres.
Section~\ref{sec:rigidity-characteristic} constructs the volume-constrained
test field from the normalized speed of Pansu dilation.
Section~\ref{sec:positive-square} establishes the constrained Hessian
identity and deduces rigidity from stability.
Section~\ref{sec:model-stability} proves the converse and the
isoperimetric consequence. Appendix~\ref{app:codazzi-graph}
derives the shape divergence identity used in the centre calculation;
Appendix~\ref{app:divergence-extension} extends weak divergence
identities across the characteristic set.

\begin{AIuse}
Starting from their existing research in Heisenberg and CR geometry, the authors formulated the geometric framework, identified the relevant mathematical structures, and provided the associated literature to ChatGPT 5.6 Sol. Through an extended, iterative interaction guided by the authors, ChatGPT 5.6 Sol assisted primarily with computations and intermediate derivations used in developing the proof of Theorem~\ref{rig:thm:stable-cmc}, while the authors supplied and interpreted their geometric meaning and determined the proof strategy. The authors independently re-derived, checked, and substantially revised all AI-assisted arguments, and take full responsibility for the correctness and content of the article.
\end{AIuse}

\section{Geometry of Pansu spheres}\label{sec:regular}
For a Euclidean unit sphere, position minus outward unit normal
is the centre. On a Pansu sphere, rotating and scaling the horizontal
normal recovers the horizontal displacement; horizontality fixes the
height. On a general hypersurface this gives a candidate centre
whose variation measures departure from the model.

\subsection{The Pansu sphere and its horizontal meridians}
\label{subsec:pansu-sphere}
The Heisenberg group $\Hn=\CC^n\times\RR$ has product
$(z,t)(w,s)=(z+w,t+s+\Im\ip{z}{w}_{\CC})$, where
$\ip{z}{w}_{\CC}=\sum_{j=1}^n z_j\overline{w_j}$ is the Hermitian product,
linear in the first argument.
Its contact form is $\Theta=dt-\Im\ip{z}{dz}_{\CC}$, and a tangent vector
is horizontal when it belongs to $\xi=\ker\Theta$.
The horizontal metric is the Euclidean inner product of the
$z$-components of horizontal vectors.

For $0\le r\le1$, define the cap height
\[
 h_{\mathrm P}(r)=\frac12\bigl(\arccos r+r\sqrt{1-r^2}\bigr).
\]
The unit Pansu ball and its boundary are
\[
 \Omega=\{(z,t):\abs z<1,\ \abs t<h_{\mathrm P}(\abs z)\},
 \qquad \PSphere=\partial\Omega.
\]
The caps meet at $|z|=1$, $t=0$ and close at
$(0,\pm\pi/4)$. Substituting $r=\cos u$ into the upper
cap and extending the height oddly across the equator gives, for
$u\in I=(-\pi/2,\pi/2)$,
\[
 r=\cos u,\qquad
 t=t_{\mathrm P}(u):=\frac u2-\frac14\Im(e^{-2iu}).
\]
The function $t_{\mathrm P}$ records the height; $u>0$ and
$u<0$ describe the upper and lower caps. To find the horizontal
curves joining the poles, let $\zeta\in\Sphere$ be their point on
the equator. Keeping the direction fixed, $z=r\zeta$, gives
$\Theta(\dot\gamma)=t_{\mathrm P}'(u)=r^2$. We therefore allow a
phase $\phi(u)$, writing $z=re^{i\phi(u)}\zeta$. The contact form gives
\[
 \Theta(\dot\gamma)=r^2\bigl(1+\phi'(u)\bigr).
\]
Thus horizontality determines $\phi(u)=-u$, with $\phi(0)=0$
fixing the equatorial label. Euler's identity
$re^{-iu}=(1+e^{-2iu})/2$ now gives
\[
 \Psi(u,\zeta)=\left(\frac{1+e^{-2iu}}2\,\zeta,
 \frac u2-\frac14\Im(e^{-2iu})\right).
\]
Differentiating along a fixed $\zeta$ gives
\[
 \partial_u z=-ie^{-2iu}\zeta,\qquad
 \partial_u t=\frac{1+\Re(e^{-2iu})}{2}
             =\frac{|1+e^{-2iu}|^2}{4}.
\]
Hence $\Theta(d\Psi(\partial_u))=0$ and $|d\Psi(\partial_u)|_H=1$.
The parameter $u$ is signed horizontal arclength from the equator,
while $\zeta$ labels the meridian. Since $t_{\mathrm P}'(u)>0$ on
$I$ and the complex factor $(1+e^{-2iu})/2$ never vanishes there,
$\Psi$ parametrizes the entire regular sphere, including the equator.
At each endpoint of $I$, all values of $\zeta$ collapse to one pole.
The pole expansion in
\cite{JacobiSpectrum}*{Proposition~5.1} shows that the sphere is
$C^2$, but not $C^3$, at the poles.

\subsection{Normal angle and horizontal shape}\label{subsec:definitions}
To identify the meridian parameter $u$ with the normal inclination,
write $z_j=x_j+iy_j$ and use the frame
\[
 X_j=\partial_{x_j}+y_j\partial_t,\qquad
 Y_j=\partial_{y_j}-x_j\partial_t,\qquad T=\partial_t.
\]
Here $T$ is the Reeb field, $[X_j,Y_j]=-2T$, and $JX_j=Y_j$, $JY_j=-X_j$.
The Levi metric is $\LTheta=\tfrac12d\Theta(\cdot,J\cdot)$; the displayed
frame is orthonormal for $\gTheta=\LTheta+\Theta^2$.
Unsubscripted inner products $\langle\cdot, \cdot \rangle$ of real tangent vectors are taken
with respect to this Riemannian metric. For horizontal vectors,
$\langle v,w\rangle=\Re\langle v,w\rangle_{\CC}$. We identify horizontal vectors with their
components in this frame.
For horizontal linear maps, we use the real Hilbert--Schmidt norm
and inner product.

Let $N$ be the outward Riemannian unit normal to $\PSphere$ and put
$N_b=N-\langle N,T\rangle T$. Wherever $N_b\ne0$, define
\[
 e_{2n}=-\frac{N_b}{|N_b|},\qquad
 e_n=-Je_{2n},\qquad \alpha=\frac{\langle N,T\rangle}{|N_b|}.
\]
Thus $e_{2n}$ is the inward horizontal unit normal, $e_n$ is tangent,
and $N/|N_b|=-e_{2n}+\alpha T$.
The horizontal tangent space is spanned
by $d\Psi(\partial_u)$ and the angular directions
$d\Psi(\CRD)$, where
\[
 \vartheta=-\Im\ip{\zeta}{d\zeta}_{\CC},\qquad \CRD=\ker\vartheta
 \quad\hbox{on }\Sphere.
\]
The subscript $S$ denotes objects on the angular sphere $\Sphere$.
Along a meridian, $e_{2n}=-e^{-2iu}\zeta$ in horizontal components,
so $e_n=-d\Psi(\partial_u)$ points toward decreasing $u$.
Tangency in the angular Hopf direction $\TSphere|_\zeta=i\zeta$
determines the vertical-to-horizontal normal ratio $\alpha$.
Since $(d\Psi(\TSphere))_b=iz$ and
$\Theta(d\Psi(\TSphere))=|z|^2=t_{\mathrm P}'(u)$, tangency gives
\[
 0=-\langle e_{2n},d\Psi(\TSphere)\rangle+\alpha\Theta(d\Psi(\TSphere))
   =-\tfrac12\Im(e^{2iu})+\alpha t_{\mathrm P}'(u).
\]
Thus $\alpha=\Im(e^{2iu})/(1+\Re(e^{2iu}))=\tan u$.
The arclength parameter is therefore also the normal's signed
inclination. We write
\begin{equation}\label{rig:eq:canonical-speeds}
 u=\arctan\alpha.
\end{equation}
Thus $\cos u=|N_b|>0$, $\sin u=\langle N,T\rangle$ and $N=-\cos u\,e_{2n}+\sin u\,T$.
On the model, $\cos u=r$ is also the horizontal radius. On a general
oriented two-sided $C^2$ hypersurface, we define $u$ from $\alpha$
by \eqref{rig:eq:canonical-speeds}; it need not be an arclength
parameter. These quantities are defined on $\Sigma^*=\{N_b\ne0\}$,
and $T+\alpha e_{2n}$ is tangent there. For boundaries we choose
the outward normal; $\{N_b=0\}$ is the characteristic set.

There are two kinds of horizontal tangent directions: the characteristic
direction $e_n$, and the transverse $J$-invariant subspace
\[
 HT\Sigma^*:=\xi\cap T\Sigma^*=\xi'\oplus\RR e_n,\qquad
 \xi'=(\xi\cap T\Sigma^*)\cap J(\xi\cap T\Sigma^*).
\]
On the sphere, $\xi'=d\Psi(\CRD)$ and $u$ is constant in these angular
directions. If $e_1,\ldots,e_{2n-1}$ is an orthonormal horizontal
tangent frame, we write
$\nabla_b\psi=\sum_a(e_a\psi)e_a$ for the horizontal tangential
gradient on the hypersurface under consideration. Hence on $\PSphere^*$,
\[
 Xu=0\quad(X\in\xi'),\qquad e_nu=-1,\qquad
 \nabla_b u=-e_n.
\]
To compare the normal's turning in the two kinds of directions,
use the Tanaka--Webster connection $\nabla$, which makes the fixed
frame parallel~\cite{CCHY}*{\S2}; in particular, $\nabla J=0$.
Define $J'$ on the horizontal tangent space by $J'|_{\xi'}=J$ and
$J'e_n=0$. The symmetric horizontal shape operator is\footnote{The
symmetric operator appears in Ritor\'e~\cite{RitoreCalibration}*{p.~52};
we use the conventions of \cite{CCHY}*{(2.8) and Proposition~2.2}.}
\[
 \Shape(X)=-\nabla_Xe_{2n}+\alpha J'X.
\]
Its trace is
the horizontal mean curvature $H_\Sigma$ (also called $p$-mean
curvature). On the model, \cite{JacobiSpectrum}*{Corollary~3.2} gives:

\begin{proposition}[Horizontal shape of the Pansu sphere]\label{prop:geometry}
On $\PSphere^*$,
\begin{equation}\label{eq:pansu-frame-shape}
 \begin{gathered}
 \alpha=\tan u,\qquad e_n=-d\Psi(\partial_u),\qquad
 e_n\alpha=-\cos^{-2} u,\\
 \Shape(e_n)=2e_n,\qquad \left.\Shape\right|_{\xi'}=\Id,
 \qquad |\Shape|^2=2n+2.
 \end{gathered}
\end{equation}
In particular, $H_{\PSphere}=2n$.
\end{proposition}

Let $d\omega$ and $\gSphere$ be the round
measure and metric on $\Sphere$, so
$dz=r^{2n-1}\,dr\,d\omega$ on $\RR^{2n}$.
Haar measure is $dV=dz\dd t$. For a hypersurface, let $d\sigma_\Sigma$
be the induced Riemannian area measure and set
$d\sigma_H=|N_b|\,d\sigma_\Sigma$. Thus, for $\Sigma=\partial\Omega$,
\[
 \PH(\Sigma)=\int_\Sigma d\sigma_H,\qquad
 V(\Omega)=\int_\Omega dV=|\Omega|.
\]
Integrals over $\Sigma$ with respect to $d\sigma_H$ of quantities
defined only on $\Sigma^*$ mean integrals over $\Sigma^*$:
the characteristic set has zero $d\sigma_H$-measure.
Integrability is verified separately below.
Since $d\Psi(\RR\partial_u\oplus\CRD)=HT\PSphere^*$, the chart gives
\begin{equation}\label{eq:pansu-pullback}
 \Psi^*\Theta=\cos^2 u\,\vartheta,\qquad
 \left.\Psi^*\LTheta\right|_{\RR\partial_u\oplus\CRD}
 =du^2+\cos^2 u\,\left.\gSphere\right|_{\CRD},\qquad
 \Psi^*d\sigma_H=\cos^{2n} u\,du\,d\omega.
\end{equation}
The metric and perimeter identities are also given in
\cite{JacobiSpectrum}*{Corollary~3.2}.
For $R>0$, $\delta_R(z,t)=(Rz,R^2t)$ gives the radius-$R$ sphere;
Section~\ref{subsec:normalization} records the corresponding scaling.

\subsection{Recovering the centre}\label{subsec:recovering-centre}
The two expressions $z=(1+e^{-2iu})\zeta/2$ and
$e_{2n}=-e^{-2iu}\zeta$ now give a geometric relation between position
and normal:
\[
 z=\frac{\zeta-e_{2n}}{2},\qquad
 \left|z-\frac{\zeta}{2}\right|=\frac12.
\]
The projected meridian runs around the circle of radius $1/2$
with centre $\zeta/2$ in $\CC\zeta$. Its radius vector $-e_{2n}/2$
turns through $-2u$, whereas the chord $z=re^{-iu}\zeta$ turns
through $-u$. This is the usual relation between central and
inscribed angles. The directions $-e_{2n}$ and $e_n=-Je_{2n}$ are therefore
the circle's radial and tangential directions. Recovering the
sphere's common centre requires the chord $z$, rather than the
meridian-dependent radius $z-\zeta/2$.

In the fixed horizontal frame, $J$ acts as multiplication by $i$.
The operator exponential consequently expresses the same rotations
without choosing a complex coordinate line:
\[
 e^{\theta J}=\cos\theta\,\Id+\sin\theta\,J,\qquad
 e^{\theta J}e^{\phi J}=e^{(\theta+\phi)J},\qquad
 (e^{\theta J})^*=e^{-\theta J}.
\]
Rotating the outward normal $-e_{2n}$ by $u$ changes its phase from $-2u$ to $-u$;
multiplication by $\cos u=r$ then restores the chord length. Conversely,
rotating the unit chord by $u$ recovers the equatorial label:
\[
 z=-\tfrac12(\Id+e^{2uJ})e_{2n}=-\cos u\,e^{uJ}e_{2n},\qquad
 e^{uJ}\frac z{\cos u}=-e^{2uJ}e_{2n}=\zeta.
\]
The Hopf tangency calculation identifies this rotation angle
with the normal inclination:
the rotation acts in $\operatorname{span}\{e_{2n},e_n\}$, and
$N=-\cos u\,e_{2n}+\sin u\,T$ describes inclination in $\operatorname{span}\{e_{2n},T\}$.

The horizontal lift in Section~\ref{subsec:pansu-sphere} fixes the
remaining coordinate: $t_{\mathrm P}'(u)=\cos^2 u$ and $t_{\mathrm P}(0)=0$.
Thus $t_{\mathrm P}$ is the signed symplectic area integral of the
projected circle, with normalization fixed by $\Theta$.
To collect the model relations, write $PX=X+\langle e_n,X\rangle e_n$:
$P$ doubles the characteristic direction and is the identity on
$\xi'$, as in Proposition~\ref{prop:geometry}. Then
\begin{equation}\label{eq:model-centre-angle}
 \nabla_b u=-e_n,\qquad \Shape=P,\qquad
 z=-\cos u\,e^{uJ}e_{2n},\qquad t=t_{\mathrm P}(u).
\end{equation}
For the left translate with centre $a=(a_z,a_t)$, they become
\begin{equation}\label{eq:model-translated-centre}
 z+\cos u\,e^{uJ}e_{2n}=a_z,\qquad
 t+\langle Ja_z,z\rangle-t_{\mathrm P}(u)=a_t.
\end{equation}
Left translation preserves the horizontal frame, so the displacement
$-\cos u\,e^{uJ}e_{2n}$ is unchanged. The group law contributes the height correction
$\langle Ja_z,z\rangle$, recovering the full centre $a$.

\subsection{Characterization by the centre map}
\label{subsec:geometric-characterization}
On a general oriented two-sided $C^2$ hypersurface, $u=\arctan\alpha$
and $e_{2n}=-N_b/|N_b|$ are determined by the normal. We apply the model
rotation and scaling to define a candidate centre on $\Sigma^*$:
\begin{equation}\label{eq:spatial-centre}
 z_0=z+\tfrac12(\Id+e^{2uJ})e_{2n}=z+\cos u\,e^{uJ}e_{2n}.
\end{equation}
The displacement still has length $|z-z_0|=\cos u>0$, even though $z_0$
may vary.
We measure the variation of the candidate centre by
\[
 D_Hz_0(X)=Xz_0\quad(X\in\xi\cap T\Sigma^*),\qquad
 |D_Hz_0|^2=\sum_{a=1}^{2n-1}|e_a z_0|_{\RR^{2n}}^2.
\]
To compare a general hypersurface with the unit Pansu sphere,
put
\begin{equation}\label{eq:shape-angle-definitions}
 \ShapeError =\Shape-P,\qquad \AngleError=\nabla_b u+e_n.
\end{equation}
The tensor $\ShapeError$ is self-adjoint and includes the mixed
components between $\xi'$ and $\RR e_n$.
Its trace is $H_\Sigma-2n$, so normalized CMC fixes only the trace,
whereas $\ShapeError=0$ prescribes the full Pansu shape operator.
The condition $\AngleError=0$ prescribes the Pansu angle variation:
$Xu=0$ for $X\in\xi'$ and $e_nu=-1$.

\Needspace{8\baselineskip}
\PansuResultBookmark{3}{prop:zero-defect-reconstruction}{Proposition}
\begin{proposition}
\label{prop:zero-defect-reconstruction}
Let $\Sigma\subset\Hn$, $n\ge2$, be a nonempty, connected, compact,
oriented embedded $C^2$ hypersurface without boundary.
Then $D_Hz_0=\AngleError=0$ on $\Sigma^*$ if and only if $\Sigma$
is a left translate of the unit Pansu sphere with its outward orientation.
\end{proposition}

We first derive the centre differential, then establish $C^1$
extension across the characteristic set and horizontal constancy.
Together these steps turn the two vanishing conditions into one fixed centre.

For a horizontal tangent field $X$, differentiating $z_0=z+\cos u\,e^{uJ}e_{2n}$
and using $\nabla_Xe_{2n}=-(\Shape-\alpha J')X$ gives directly
\begin{equation}\label{rig:eq:centre-defect-differential}
 Xz_0=X-\cos u\,e^{uJ}(\Shape-\alpha J')X-(Xu)e^{2uJ}e_n
     =-\cos u\,e^{uJ}\ShapeError X-\langle\AngleError,X\rangle e^{2uJ}e_n.
\end{equation}
Here $\partial_u(\cos u\,e^{uJ})e_{2n}=-e^{2uJ}e_n$ and
$\cos u\,e^{uJ}(P-\alpha J')X=X+\langle X,e_n\rangle e^{2uJ}e_n$;
the latter follows from $\cos u\,(P-\alpha J')=e^{-uJ}$ on $\xi'$
and $2\cos u\,e^{uJ}e_n=e_n+e^{2uJ}e_n$.

Let $(E_a)_{a=1}^{2n-1}$ be a horizontal tangent orthonormal frame.
Applying the orthogonal map $e^{-uJ}$ in \eqref{rig:eq:centre-defect-differential}
and using $e^{uJ}e_n=\cos u\,e_n+\sin u\,e_{2n}$, we obtain
\begin{equation}\label{eq:centre-differential-square}
 \begin{aligned}
 |D_Hz_0|^2
 &=\sum_a\bigl|\cos u\,\ShapeError E_a
       +\langle\AngleError,E_a\rangle e^{uJ}e_n\bigr|^2\\
 &=\cos^2 u\,|\ShapeError|^2+|\AngleError|^2
   +2\cos u\,\sum_a\langle\AngleError,E_a\rangle
           \langle\ShapeError E_a,\cos u\,e_n+\sin u\,e_{2n}\rangle\\
 &=|\cos u\,\ShapeError|^2+|\AngleError|^2
   +2\cos^2 u\,\langle\AngleError,\ShapeError e_n\rangle.
 \end{aligned}
\end{equation}
In the last step, $\ShapeError E_a\perp e_{2n}$ removes the normal term,
while $\sum_a\langle\AngleError,E_a\rangle E_a=\AngleError$
collapses the sum. Self-adjointness then gives
$\langle\ShapeError\AngleError,e_n\rangle
=\langle\AngleError,\ShapeError e_n\rangle$.

The vertical correction is determined by the same model height and
left-translation law. Complete the candidate centre by setting
\begin{equation}\label{rig:eq:t-centre}
 t_0=t+\langle Jz_0,z\rangle-t_{\mathrm P}(u).
\end{equation}
For every horizontal tangent field $X$ on $\Sigma^*$,
\begin{equation}\label{rig:eq:time-centre-differential}
 Xt_0=\langle J(Xz_0),z\rangle
       -\cos^2 u\,\langle \AngleError,X\rangle.
\end{equation}
Indeed, horizontality gives $Xt=-\langle Jz,X\rangle$.
The half-angle relation and the model lift give
\[
 (J(z-z_0))^{\top_H}=\cos^2 u\,e_n,\qquad t_{\mathrm P}'(u)=\cos^2 u,
\]
where $\top_H$ denotes horizontal tangent projection.
Differentiating \eqref{rig:eq:t-centre} therefore leaves
$Xt_0-\langle J(Xz_0),z\rangle
=-\cos^2 u\,(Xu+\langle e_n,X\rangle)$, as required.

For the extension and constancy arguments, note that the regular set
is dense: otherwise,
$\Theta|_\Sigma=0$ on an open set, and taking its exterior derivative
would give $d\Theta|_{T\Sigma}=0$ there, contrary to the
nondegeneracy of $d\Theta|_\xi$.

\Needspace{8\baselineskip}
\begin{lemma}[Extension of the centre map]\label{lem:centre-extension}
On every oriented two-sided $C^2$ hypersurface, the candidate centre
$(z_0,t_0)$ extends in $C^1$ across the characteristic set.
\end{lemma}
\begin{proof}
Define $a_0(x)=\arcsin x+x\sqrt{1-x^2}$ for $|x|\le1$,
and continue $a_0$ constantly beyond the endpoints.
On $\Sigma^*$, we have $\cos u=|N_b|$, $\cos u\,e_{2n}=-N_b$
and $\sin u=\langle N,T\rangle$, so
\[
 \begin{aligned}
 \cos u\,e^{uJ}e_{2n}
   &=(\cos u\,\Id+\sin u\,J)(\cos u\,e_{2n})\\
   &=-\cos u\,N_b-\sin u\,JN_b,\\
 a_0(\sin u)&=u+\sin u\cos u
             =u+\tfrac12\sin(2u)=2t_{\mathrm P}(u).
 \end{aligned}
\]
Thus the original centre formulas become
\[
 \begin{aligned}
 z_0&=z-|N_b|N_b-\langle N,T\rangle JN_b,\\
 t_0&=t+\langle Jz_0,z\rangle-\tfrac12a_0(\langle N,T\rangle).
 \end{aligned}
\]
To check their regularity, put $R(v)=|v|v$. For $v\ne0$,
\[
 DR(v)[w]=|v|w+\frac{\langle v,w\rangle}{|v|}v,
 \qquad |DR(v)[w]|\le2|v||w|.
\]
Since $|R(v)-R(0)|=|v|^2$, we have $DR(0)=0$; the displayed
bound shows $DR(v)\to0$ as $v\to0$. Hence $R$ is $C^1$.
For $|x|<1$, direct differentiation gives
\[
 a_0'(x)=\frac1{\sqrt{1-x^2}}+\sqrt{1-x^2}
         -\frac{x^2}{\sqrt{1-x^2}}
        =2\sqrt{1-x^2}\longrightarrow0\quad(x\to\pm1).
\]
The constant continuation therefore makes $a_0$ a $C^1$ function.
Both centre formulas are consequently $C^1$ in the position and normal,
including at $N_b=0$, $\langle N,T\rangle=\pm1$. Composition with the $C^1$ unit
normal gives the required extension.
\end{proof}

\Needspace{5\baselineskip}
\begin{lemma}[Global horizontal constancy]\label{lem:horizontal-constancy}
Let $\Sigma$ be a connected $C^2$ hypersurface in $\Hn$, $n\ge2$.
If $\psi\in C^1(\Sigma)$ and $d\psi$ annihilates the horizontal tangent
space on $\Sigma^*$, then $\psi$ is constant.
\end{lemma}
\begin{proof}
Put $\theta=\Theta|_\Sigma$, a $C^1$ form. On $\Sigma^*$ the hypothesis
gives $d\psi\wedge\theta=0$. Smoothly approximating $\psi$ and $\theta$
in local coordinates and passing to the distributional limit in the
exterior-derivative product rule gives
$0=d(d\psi\wedge\theta)=-d\psi\wedge d\theta$.
Writing $d\psi=a\theta$ with $a$ continuous now gives
$a\theta\wedge d\theta=0$.
On $HT\Sigma^*=\xi'\oplus\RR e_n$, the contact form gives
\[
 d\theta(X,JX)=2|X|^2\quad(X\in\xi'),\qquad
 d\theta(e_n,Y)=2\langle e_{2n},Y\rangle=0\quad(Y\in HT\Sigma^*).
\]
Since $\xi'$ is $J$-invariant, the restriction to $\xi'$ is nondegenerate.
Thus the kernel on $HT\Sigma^*$ is $\RR e_n$, and the rank is $2n-2>0$.
At each regular point
choose horizontal tangent vectors $X,Y$ with $d\theta(X,Y)\ne0$
and a tangent vector $Z$ with $\theta(Z)=1$. Then
$(\theta\wedge d\theta)(Z,X,Y)=d\theta(X,Y)\ne0$.
Thus $a=0$ and $d\psi=0$ on $\Sigma^*$.
Density and continuity extend $d\psi=0$ to $\Sigma$, and connectedness
proves constancy.
\end{proof}

\begin{proof}[Proof of Proposition~\ref{prop:zero-defect-reconstruction}]
Assume $D_Hz_0=\AngleError=0$.
Equation~\eqref{rig:eq:time-centre-differential} gives
$Xt_0=0$ for every horizontal tangent field $X$.
The $C^1$ centre extension in Lemma~\ref{lem:centre-extension} and
global horizontal constancy (Lemma~\ref{lem:horizontal-constancy}),
applied componentwise, give one centre $a=(a_z,a_t)$.

To identify the whole hypersurface, use $|z-z_0|=\cos u$: the centre equations give on $\Sigma^*$
\begin{equation}\label{rig:eq:pansu-profile}
 |z-a_z|=\cos u,\qquad
 t+\langle Ja_z,z\rangle-a_t
 =t_{\mathrm P}(u).
\end{equation}
Thus $\Sigma^*$ lies in the unit Pansu sphere $\mathcal S_a$;
density and closedness give $\Sigma\subset\mathcal S_a$.
The inclusion between these $C^2$ embedded hypersurfaces of equal
dimension has invertible differential, so it is a local diffeomorphism.
Its image is open and, by compactness, closed in connected
$\mathcal S_a$; hence $\Sigma=\mathcal S_a$.
The height in \eqref{rig:eq:pansu-profile} is strictly increasing in $u$.
It fixes $\alpha=\tan u$ to be the outward Pansu sphere value on the
regular caps, and continuity fixes the orientation on the whole sphere.

Conversely, on an outward-oriented left translate of the unit Pansu sphere,
Proposition~\ref{prop:geometry} and \eqref{eq:spatial-centre} give
$\AngleError=0$ and $z_0$ constant.
\end{proof}

\begin{applicationremark}[Regularity and earlier classifications]
For comparison, \cite{CCHY}*{Theorem~1.3} identifies closed connected
umbilic hypersurfaces with $l=2k$ in their notation as Pansu spheres. Both
\cite{CCHY} and \cite{CCHYSigma} assume $C^2$ hypersurfaces whose
noncharacteristic parts are $C^\infty$; the umbilic classification in
\cite{CLL}*{Corollary~4.4} uses \cite{CCHYSigma}*{Theorem~B}.
Proposition~\ref{prop:zero-defect-reconstruction} establishes the
reconstruction at the stated $C^2$ regularity.
\end{applicationremark}

\begin{applicationremark}[Complex coordinates and CR sphere geometry]
The horizontal position and contact height combine into
$q=|z|^2-2it$: the map $(z,t)\mapsto(z,q)$ realizes $\Hn$ as
the hypersurface $\Re q=|z|^2$ in $\CC^{n+1}$.
Complex reciprocation and the Cayley transform then take the
forms in Table~\ref{tab:cr-sphere-geometry}; see
\cite{BFM}*{pp.~9, 15}, with the contact normalization
used here. The meridian formulas concern the unit Pansu sphere
centred at the origin, with $\zeta$ fixed and $\cos u>0$. Put
$f(u)=1+u\tan u$; a dot denotes
differentiation in $u$, so $\dot f=\tan u+\frac{u}{\cos^2 u}$.

In $(z,q)$ coordinates, $\mathcal I$ is the holomorphic map
$(z,q)\mapsto(z/q,1/q)$ preserving $\Re q=|z|^2$.
The Cayley transform identifies $\Hn$ with the standard CR sphere
minus $(0,-1)$; under this identification, inversion becomes
the unitary reflection $(\eta,\tau)\mapsto(\eta,-\tau)$.
For the inverted Pansu meridian, Euler's identity gives
\[
 q\circ\Psi=\cos^2 u -2it_{\mathrm P}
 =\cos u\,e^{-iu}(f-iu),\qquad \frac{2t_{\mathrm P}}{\cos^2 u}=\dot f.
\]
Thus the same function $f(u)=F(\tan u)$ that appears as the
normalized dilation speed in Section~\ref{subsec:angle-profile} also
governs the inverted meridian. Because $\mathcal I$ is CR, its contact
factor $|q|^{-2}$ also scales the horizontal metric. The inverted
meridian is therefore horizontal, with squared speed $\cos^{-2} u\,(f^2+u^2)^{-1}$.
\end{applicationremark}

\PansuResultBookmark{2}{tab:cr-sphere-geometry}{Table}
\begin{table}[!htbp]
\caption{Complex coordinates and CR sphere geometry}
\label{tab:cr-sphere-geometry}
\centering
\renewcommand{\arraystretch}{1.15}
\setlength{\tabcolsep}{8pt}
\begin{tabularx}{\linewidth}{@{}>{\raggedright\arraybackslash}p{0.22\linewidth}
 >{\raggedright\arraybackslash}X@{}}
\toprule
\textsc{Geometry} & \textsc{Identities} \\
\midrule
\addlinespace[3pt]
Half-angle and lift &
 $\begin{aligned}[t]
 z&=\tfrac12(\zeta-e_{2n})=-\cos u\,e^{uJ}e_{2n},\\[3pt]
 e_{2n}&=-e^{-2iu}\zeta,\qquad
 \dot t_{\mathrm P}=-\langle Jz,\dot z\rangle=\cos^2 u.
 \end{aligned}$ \\
\addlinespace[10pt]
CR inversion &
 $\begin{aligned}[t]
 \mathcal I(z,t)&=\left(\frac zq,-\frac{t}{|q|^2}\right),\qquad q\ne0,\\[5pt]
 q\circ\mathcal I&=q^{-1},\qquad \mathcal I^*\Theta=|q|^{-2}\Theta.
 \end{aligned}$ \\
\addlinespace[10pt]
Cayley transform &
 $\displaystyle \mathcal C(z,t)=
 \left(\frac{2z}{1+q},\frac{1-q}{1+q}\right)
 \in\mathbb S^{2n+1}\subset\CC^{n+1}.$ \\
\addlinespace[10pt]
Inverted Pansu meridian &
 $\displaystyle (\widetilde z,\widetilde t)
 =\mathcal I\circ\Psi(u,\zeta)
 =\left(\frac{\zeta}{f-iu},-\frac{\dot f}{2(f^2+u^2)}\right).$ \\
\addlinespace[3pt]
\bottomrule
\end{tabularx}
\par\smallskip
\begin{minipage}{\linewidth}
Here $q=|z|^2-2it$, $f=1+u\tan u$,
and $\dot f=df/du$.
\end{minipage}
\end{table}
\FloatBarrier

\section{The volume-constrained test field}\label{sec:rigidity-characteristic}
To test the centre and angle conditions of
Proposition~\ref{prop:zero-defect-reconstruction}, we transfer the
normal speed of Pansu dilation to a general boundary, realize it by
an ambient field and subtract a multiple of $D$ to impose the linearized
volume constraint. The Hessian in
Section~\ref{sec:positive-square} will control $D_Hz_0$ and $\AngleError$.

In Sections~\ref{sec:rigidity-characteristic}--\ref{sec:positive-square},
let $\Omega\subset\Hn$ be a bounded nonempty open set with $C^2$
boundary $\Sigma=\partial\Omega$ and outward unit normal $N$.
Laplacian and divergence identities are understood weakly on $\Sigma^*$;
their displayed scalar densities are continuous there.
For background on sub-Riemannian hypersurface calculus and variational
formulas, see \cites{CDPT,DGN,MontefalconeVar}; for the graph
variational problem, see \cite{CHY}.

\subsection{Dilation and normalization}\label{subsec:normalization}
\label{subsec:regular-variation}
We first derive CMC from stationarity and use dilation to normalize
the mean curvature. Let $\Sigma_\tau=\Phi_\tau(\Sigma)$,
where $\Phi_0=\Id$ and $\Phi_\tau$ is a diffeomorphism of a fixed
neighborhood for small $\tau$, with initial velocity $W$.
Its normalized normal speed and volume flux are
\[
 \varphi_W=\frac{\langle W,N\rangle}{|N_b|},\qquad
 \delta V[W]=\int_\Sigma\varphi_W\,d\sigma_H.
\]
The physical speed is $\langle W,N\rangle=\cos u\,\varphi_W$ and
$d\sigma_H=\cos u\,d\sigma_\Sigma$, so either convention gives the same
volume flux. The horizontal representative
$-\varphi_W e_{2n}$ has normalized speed $\varphi_W$
on $\Sigma^*$. Write $\delta\mathcal J[W]$ for the first variation
of $\mathcal J$ in direction $W$.
For $h\in\RR$, put $\mathcal F_h=\PH-hV$. For $C^2$ surface fields
$U,X$, its affine second variation is the mixed derivative
\[
 \mathcal B_h(U,X)
 =\left.\partial_r\partial_s\right|_{r=s=0}
   \mathcal F_h\bigl((\Id+rU+sX)\Sigma\bigr),
\]
computed in the fixed coordinates $(z,t)$.
Differentiation means pulling the moving perimeter density
back to a fixed parameter domain and differentiating with its
reference point held fixed.
In an oriented surface chart, the horizontal cofactor is $N_b$
times the surface area density; its norm is the horizontal perimeter
density. At characteristic points the differential of the horizontal
cofactor has rank at least $n$. For $n\ge2$, the resulting sublevel
estimate makes the reciprocal cofactor norm locally integrable
\cite{JacobiSpectrum}*{Corollary~2.2}. This controls the singular
factor in the second density and justifies differentiation
for smooth ambient families, flows of $C^2$ ambient fields, and affine
$C^{1,1}$ surface variations
\cite{JacobiSpectrum}*{Proposition~2.6}.
The first-variation formula~\cite{JacobiSpectrum}*{Proposition~2.7} gives
\[
 \delta\mathcal F_h[W]=\int_\Sigma(H_\Sigma-h)\varphi_W\,d\sigma_H.
\]
Hence $H_\Sigma=h$ on $\Sigma^*$ makes $\mathcal F_h$ stationary
under ambient variations;
$\mathcal B_h$ is then the symmetric Hessian of $\mathcal F_h$,
and the second variation depends only on the initial velocity,
since the acceleration term is a first variation of $\mathcal F_h$.

The dilation field determines the multiplier and allows us to impose
the linearized volume constraint. It satisfies
\[
 D=z\cdot\partial_z+2t\partial_t,\qquad
 \delta V[D]=(2n+2)|\Omega|>0.
\]

\begin{lemma}
\label{lem:stationarity-cmc}
The boundary is volume-preserving stationary if and only if
$H_\Sigma=h$ on $\Sigma^*$ for one constant $h$. In that case
\[
 h=\frac{(2n+1)\PH(\Sigma)}{(2n+2)|\Omega|}>0,
 \qquad \delta\PH[X]=h\delta V[X]
\]
for every $C^2$ ambient field $X$, with the same $h$ on all boundary
components. It is volume-preserving stable if and only if, in
addition, $\mathcal B_h(X,X)\ge0$ for every such $X$ with
$\delta V[X]=0$.
\end{lemma}
\begin{proof}
Let $U$ be a smooth field with $\delta V[U]=0$, and let $\Phi_t$
be its flow. Set
\[
 v(t)=|\Phi_t(\Omega)|,\qquad
 \lambda(t)=\left(\frac{|\Omega|}{v(t)}\right)^{1/(2n+2)},\qquad
 \Psi_t=\delta_{\lambda(t)}\circ\Phi_t.
\]
Volume homogeneity and $v'(0)=\delta V[U]$ give
\[
 \begin{aligned}
 |\Psi_t(\Omega)|&=\lambda(t)^{2n+2}v(t)=|\Omega|,\\
 \lambda(0)&=1,\qquad
 \lambda'(0)=-\frac{v'(0)}{(2n+2)|\Omega|}
             =-\frac{\delta V[U]}{(2n+2)|\Omega|}=0.
 \end{aligned}
\]
Thus $\Psi_t$ preserves volume exactly and has initial velocity
$\dot\Psi_0=U+\lambda'(0)D=U$.
Stationarity therefore gives $\delta\PH[U]=0$.
For any smooth $X$, apply this to
$U=X-\delta V[X]D/\delta V[D]$ to obtain
$\delta\PH[X]=h\delta V[X]$, with
$h=\delta\PH[D]/\delta V[D]$.
Dilation homogeneity gives the displayed positive value of $h$.
The global first-variation formula, tested on arbitrary regular
patches, gives $H_\Sigma=h$ there. Conversely, if $H_\Sigma=h$ on
$\Sigma^*$, the same formula gives $\delta\PH[U]=h\delta V[U]=0$
for every smooth field $U$ with $\delta V[U]=0$.

At this stationarity, let $\Psi_t$ be any smooth exactly
volume-preserving variation with initial velocity $U$.
Since its volume is constant and $\delta\mathcal F_h=0$,
differentiation gives
\[
 \begin{aligned}
 \left.\frac{d^2}{dt^2}\right|_{0}\PH(\Psi_t(\Sigma))
 &=\left.\frac{d^2}{dt^2}\right|_{0}\mathcal F_h(\Psi_t(\Sigma))\\
 &=\mathcal B_h(U,U)+\delta\mathcal F_h[\ddot\Psi_0]
  =\mathcal B_h(U,U).
 \end{aligned}
\]
The preceding adjustment realizes every smooth field satisfying
the linearized volume constraint. Smooth approximation, subtraction of its volume
flux times $D/\delta V[D]$, and the $C^1$ continuity in
\cite{JacobiSpectrum}*{Proposition~2.6} extend the equivalence
to all $C^2$ ambient fields. The same continuity extends the
first-variation identity.
\end{proof}

\begin{samepage}
\PansuResultBookmark{3}{prop:positive-connected}{Proposition}
\begin{proposition}[Dilation and the volume constraint]
\label{prop:positive-connected}
Under the boundary assumptions above, let $H_\Sigma=h$ on $\Sigma^*$.
Then, for every $C^2$ ambient field $X$,
\begin{equation}\label{eq:volume-hessian-orthogonality}
 \mathcal B_h(D,X)=-h\,\delta V[X].
\end{equation}
In particular, $\mathcal B_h(D,D)=-h\delta V[D]<0$ and
$\ker\delta V=\{X:\mathcal B_h(D,X)=0\}$:
the fields satisfying the linearized volume constraint form the
Hessian orthogonal complement of $D$.
\end{proposition}
\end{samepage}
\begin{proof}
For a $C^2$ field $X$, consider
$\Phi_{\tau,s}=\delta_{e^\tau}\circ(\operatorname{Id}+sX)$.
Perimeter and volume homogeneity give
\[
 \begin{aligned}
 \left.\partial_\tau\partial_s\right|_{(0,0)}
       \mathcal F_h(\Phi_{\tau,s}(\Sigma))
 &=(2n+1)\delta\PH[X]-h(2n+2)\delta V[X]\\
 &=h\bigl((2n+1)-(2n+2)\bigr)\delta V[X]
  =-h\,\delta V[X].
 \end{aligned}
\]
The initial velocities are $D$ and $X$, and the mixed acceleration
is $(\mathrm dD)X$, where $\mathrm dD$ is the ordinary coordinate
derivative. The mixed dilation formula in the proof of
\cite{JacobiSpectrum}*{Proposition~2.7}
expresses the same mixed derivative as
\[
 \mathcal B_h(D,X)+\delta\mathcal F_h[(\mathrm dD)X]
 =\mathcal B_h(D,X),
\]
because stationarity makes the first-variation term zero.
Equating the two expressions proves \eqref{eq:volume-hessian-orthogonality}.
Taking $X=D$ gives the negative dilation direction.
\end{proof}

Write $\operatorname{ind}_{\mathrm{vol}}$ for the index on $C^2$
ambient fields satisfying the linearized volume constraint.
Every ambient field has the unique decomposition
$X=\delta V[X]D/\delta V[D]+X_0$, with $X_0\in\ker\delta V$.
The two summands are Hessian-orthogonal, so
\begin{equation}\label{eq:dilation-orthogonal-splitting}
 \mathcal B_h(X,X)=\mathcal B_h(X_0,X_0)
 -\frac{h}{\delta V[D]}\bigl(\delta V[X]\bigr)^2.
\end{equation}
Adjoining $D$ raises a negative subspace of $\ker\delta V$ by one
dimension, while intersection with $\ker\delta V$ lowers any negative
subspace by at most one. Taking suprema gives
\begin{equation}\label{eq:ambient-volume-index}
 \operatorname{ind}(\mathcal B_h)
 =\operatorname{ind}_{\mathrm{vol}}(\mathcal B_h)+1,
\end{equation}
with $\infty+1=\infty$. Thus ambient index one is equivalent to
volume-preserving stability.

Set $\lambda=h/(2n)>0$ and
$(\widetilde\Sigma,\widetilde\Omega)=\delta_\lambda(\Sigma,\Omega)$.
For corresponding variations, homogeneity gives
\[
 \PH(\widetilde\Sigma_\tau)-2nV(\widetilde\Omega_\tau)
 =\lambda^{2n+1}
   \bigl(\PH(\Sigma_\tau)-2n\lambda V(\Omega_\tau)\bigr).
\]
Since $2n\lambda=h$, stationarity and the first-variation formula
give $H_{\widetilde\Sigma}=\lambda^{-1}H_\Sigma=2n$.
Volume flux scales by $\lambda^{2n+2}$ and the
Hessian by the positive factor $\lambda^{2n+1}$, so stability is
equivalent before and after dilation. Boundary regularity and the
characteristic set are preserved. We henceforth work at $H_\Sigma=2n$,
omit tildes, and write
$\mathcal B=\mathcal B_{2n}$.

\subsection{The normalized speed of Pansu dilation}
\label{subsec:angle-profile}
On the unit Pansu sphere centred at the origin, the fields $T$ and
$D$ have normalized speeds
\begin{equation}\label{eq:dilation-support-speed}
 \varphi_T=\alpha=\tan u,\qquad
 \varphi_D=\cos^2 u+2\alpha t_{\mathrm P}(u)=1+u\alpha.
\end{equation}
Here $\Theta(D)=2t_{\mathrm P}(u)$, while its horizontal part satisfies
$-\langle z,e_{2n}\rangle=\cos u\,\langle e^{uJ}e_{2n},e_{2n}\rangle=\cos^2 u$.
Thus $F(\alpha)=1+\alpha\arctan\alpha$ is the normalized normal
speed of dilation on the unit Pansu sphere. On a general boundary,
we use the same function of the normal angle; it need not equal
the normalized dilation speed there.

Define the modified horizontal tangential sub-Laplacian $\Delta_b$ by
\[
 \int_{\Sigma}\ip{\nabla_b \psi}{\nabla_b \phi}\dd\sigma_H
 =-\int_{\Sigma}(\Delta_b \psi)\phi\dd\sigma_H
\]
for tests $\phi\in C_c^1(\Sigma^*)$, so $\Delta_b $ is nonpositive.
For any tangential field, $\operatorname{div}_H$ denotes divergence
with respect to $d\sigma_H$; the $H$ refers to the measure.
Thus $\Delta_b=\operatorname{div}_H\nabla_b$.\footnote{The modifier denotes the drift needed for symmetry in
$d\sigma_H$. Under $(z,t)\mapsto(z,-t/2)$, our horizontal frame agrees
with Montefalcone's; our horizontal divergence and $\Delta_b$ are his
$\mathcal D_{HS}$ and $\mathcal L_{HS}$
\cite{MontefalconeStability}*{Definition~1.16}.
The latter includes the drift absent from $\Delta_{HS}$.}

Set $L=-\Delta_b-q_\Sigma$, with potential
\begin{equation}\label{eq:potential-symmetric-shape}
 q_{\Sigma}=\abs{\Shape}^2+4e_n\alpha+2(n+1)\alpha^2.
\end{equation}
This is the potential in Giovannardi, Pinamonti and
Verzellesi~\cite{GPV}*{Section~3.8, equation~(3.36)}.
In the same horizontal frame, their outward horizontal normal $\nu$ equals
$-e_{2n}$ and its $J$-rotation is $e_n$; their symmetrized second fundamental
form is represented by $\Shape$.

For speeds compactly supported in $\Sigma^*$, the regular Jacobi pairing is
\begin{equation}\label{eq:general-jacobi-pairing}
 Q_\Sigma(\varphi,\psi)=\int_\Sigma
 (\langle\nabla_b \varphi,\nabla_b \psi\rangle-q_\Sigma\varphi\psi)\,d\sigma_H.
\end{equation}
On smooth CMC patches, this is the Hessian of $\PH-hV$ ($H_\Sigma=h$); see
\cite{HP}*{Lemma~7.2 and Section~8.2}, expressed in our frame normalization
and the symmetric shape convention of
\cite{CCHY}*{Propositions~2.1 and~2.2}.
The local pairing at $C^2$ regularity is given by
\cite{JacobiSpectrum}*{Proposition~2.9}.
This regular pairing does not by itself identify the Hessian of
a field that moves characteristic points. For the realization $X_F$
on a general CMC boundary, Section~\ref{subsec:global-centre-pairing}
will prove $\mathcal B(X_F,X_F)=\int F\,LF\,d\sigma_H$ by
differentiation.
On the Pansu sphere, the pole estimates and global energy representation
recalled in Section~\ref{subsec:form-domain} identify the Hessian for
every $C^1$ ambient field.

Primes of $F$ denote differentiation with respect to $\alpha$.

\begin{samepage}
\begin{proposition}[Jacobi identity for \texorpdfstring{$F$}{F}]
\label{lem:central-jacobi}\label{lem:canonical-multiplier}
At $H_\Sigma=2n$, the global Hessian satisfies $\mathcal B(T,X)=0$
for every $C^1$ surface field $X$. The translation speed $\alpha$ satisfies
\begin{equation}\label{eq:central-jacobi}
 L\alpha=0\quad\text{weakly on }\Sigma^*.
\end{equation}
The derivatives of $F$ are determined by the model height evaluated
at $u=\arctan\alpha$:
\begin{equation}\label{rig:eq:support-derivatives}
 F(\alpha)=1+\alpha\arctan\alpha,\qquad
 F'(\alpha)=2t_{\mathrm P}(u),\qquad F''(\alpha)=2\cos^4 u,
\end{equation}
and its Jacobi image is
\begin{equation}\label{rig:eq:LF-direct}
 LF=-\cos^2 u\,q_\Sigma-2|\nabla_b u|^2.
\end{equation}
\end{proposition}
\end{samepage}
\begin{proof}
Translation invariance and stationarity give $\mathcal B(T,X)=0$
for smooth $X$. The $C^1$ continuity of $\mathcal B$ in
\cite{JacobiSpectrum}*{Proposition~2.6} extends this identity to all
$C^1$ surface fields. The weak equation for $\alpha$ is
\cite{JacobiSpectrum}*{Lemma~2.8(ii)}.
Since $F'(\alpha)=u+\alpha \cos^2 u=2t_{\mathrm P}(u)$ and
$du/d\alpha=\cos^2 u$, direct differentiation gives
$F''(\alpha)=2\cos^4 u$. The dilation identity
\eqref{eq:dilation-support-speed} also gives $F-\alpha F'=\cos^2 u$.
Since $\alpha\in C^1(\Sigma^*)$, $F'(\alpha)\psi$ is admissible
for every $\psi\in C_c^1(\Sigma^*)$. Testing $L\alpha=0$ with this
function and using $\nabla_bF'=F''\nabla_b\alpha$ gives
\[
 \begin{aligned}
 0&=\int_\Sigma\bigl(
   \langle\nabla_b\alpha,\nabla_b(F'\psi)\rangle
   -q_\Sigma\alpha F'\psi\bigr)\,d\sigma_H\\
 &=\int_\Sigma F'\langle\nabla_b\alpha,\nabla_b\psi\rangle\,d\sigma_H
   +\int_\Sigma\bigl(F''|\nabla_b\alpha|^2
   -q_\Sigma\alpha F'\bigr)\psi\,d\sigma_H.
 \end{aligned}
\]
Since $\nabla_bF=F'\nabla_b\alpha$, it follows that
\[
 \begin{aligned}
 \int_\Sigma\bigl(\langle\nabla_bF,\nabla_b\psi\rangle
   -q_\Sigma F\psi\bigr)\,d\sigma_H
 &=\int_\Sigma\bigl(q_\Sigma\alpha F'
   -F''|\nabla_b\alpha|^2-q_\Sigma F\bigr)\psi\,d\sigma_H\\
 &=\int_\Sigma\bigl((\alpha F'-F)q_\Sigma
   -F''|\nabla_b\alpha|^2\bigr)\psi\,d\sigma_H.
 \end{aligned}
\]
Thus the weak chain rule and $\nabla_bu=\cos^2 u\,\nabla_b\alpha$ give
\[
 \begin{aligned}
 L(F(\alpha))
 &=(\alpha F'-F)q_\Sigma-F''|\nabla_b\alpha|^2\\
 &=-\cos^2 u\,q_\Sigma-2\cos^4 u\,|\nabla_b\alpha|^2
  =-\cos^2 u\,q_\Sigma-2|\nabla_bu|^2,
 \end{aligned}
\]
which proves \eqref{rig:eq:LF-direct} using only first derivatives
of $\alpha$.
\end{proof}

\subsection{Ambient realization across characteristic points}
\label{subsec:admissibility}
Although $F$ diverges as $\cos u\to0$, its physical speed $\cos u\,F=\cos u+u\sin u$
stays bounded. A vertical representative near the characteristic
set realizes it by a $C^1$ ambient field. This regularity suffices:
for every $h\in\RR$, the affine second variation extends uniquely
to $C^1$ surface fields by \cite{JacobiSpectrum}*{Proposition~2.6}, with
\[
 |\mathcal B_h(U,X)|\le C_{\Sigma,h}\|U\|_{C^1}\|X\|_{C^1}.
\]

\begin{lemma}[Ambient realization of \texorpdfstring{$F$}{F}]
\label{rig:lem:characteristic-admissibility}
Let $\chi:[0,1]\to[0,1]$ be smooth, equal to one on $[0,1/2]$
and zero on $[2/3,1]$. On every compact oriented $C^2$ hypersurface,
the surface field
\[
 X_F=
 \begin{cases}
 \displaystyle
 \chi(\cos u)\frac{F(\alpha)}{\alpha}T
 -(1-\chi(\cos u))F(\alpha)e_{2n},&\text{on }\Sigma^*,\\[4pt]
 \displaystyle
 \frac{\pi}{2}\operatorname{sgn}\langle N,T\rangle\,T,
 &\text{on }\Sigma\setminus\Sigma^*
 \end{cases}
\]
extends to a global $C^1$ ambient field, vertical near the characteristic
set, with $\varphi_{X_F}=F(\alpha)$ on $\Sigma^*$.
The quotient term is set to zero where $\chi(\cos u)=0$;
where it is used, $\cos u<2/3$ and hence $\alpha\ne0$.
\end{lemma}
\begin{proof}
Near a characteristic point, write the hypersurface as a vertical graph
$t=f(z)$ and put
\[
 \rho(z,t)=t-f(z),\qquad G=\nabla f+Jz,\qquad g=|G|.
\]
In the fixed horizontal frame, $Jz=(-y,x)$, so
\[
 X_j\rho=-f_{x_j}+y_j=-G_{x_j},\qquad
 Y_j\rho=-f_{y_j}-x_j=-G_{y_j},\qquad T\rho=1.
\]
Here $G_{x_j}$ and $G_{y_j}$ denote components of $G$.
Since the frame is orthonormal, the chosen unit normal and its horizontal
part are
\[
 N=\frac{\sigma(T-G)}{\sqrt{1+g^2}},\qquad
 N_b=-\frac{\sigma G}{\sqrt{1+g^2}},\qquad
 |N_b|=\frac{g}{\sqrt{1+g^2}},
\]
where $\sigma=\operatorname{sgn}\langle N,T\rangle\in\{\pm1\}$ is
constant after shrinking the graph patch. Hence, on $\{g>0\}$,
\[
 \alpha=\frac{\langle N,T\rangle}{|N_b|}
 =\frac{\sigma/\sqrt{1+g^2}}{g/\sqrt{1+g^2}}
 =\frac{\sigma}{g},\qquad
 \varphi_{aT}=\frac{\langle aT,N\rangle}{|N_b|}
 =a\alpha=\frac{\sigma a}{g}.
\]
Thus $\varphi_{aT}=F(\alpha)$ precisely when
$a=\sigma gF(\sigma/g)$. Using $F(\alpha)=1+\alpha\arctan\alpha$,
the oddness of $\arctan$, and
$\arctan(1/g)=\pi/2-\arctan g$ for $g>0$, take
\begin{equation}\label{rig:eq:canonical-heights}
 \begin{aligned}
 a_F&=\sigma gF\!\left(\frac{\sigma}{g}\right)
 =\sigma\left(g+\sigma\arctan\frac{\sigma}{g}\right)\\
 &=\sigma\left(g+\arctan\frac1g\right)
 =\sigma\left(\frac{\pi}{2}+g-\arctan g\right).
 \end{aligned}
\end{equation}
Set $a_F=\sigma\pi/2$ where $G=0$. The cancellation and the chain rule give
\[
 \begin{aligned}
 a_F-\frac{\sigma\pi}{2}
 &=\sigma\int_0^g\left(1-\frac1{1+r^2}\right)\,dr
 =\sigma\int_0^g\frac{r^2}{1+r^2}\,dr=O(g^3),\\
 D_za_F
 &=\sigma\left(1-\frac1{1+g^2}\right)D_zg
 =\frac{\sigma g\,G^{\mathsf T}D_zG}{1+g^2}=O(g^2)
 \qquad(g>0).
 \end{aligned}
\]
Here $D_zg=G^{\mathsf T}D_zG/g$ and $D_zG$ is locally bounded since
$f\in C^2$. At every $z_*$ with $G(z_*)=0$, local Lipschitz continuity
of $G$ gives $g(z)=O(|z-z_*|)$, hence
$a_F(z)-a_F(z_*)=O(|z-z_*|^3)$. Thus $D_za_F(z_*)=0$, and the
$O(g^2)$ bound makes the derivative continuous there.
Consequently $a_F(z)T$ is a $C^1$ ambient field realizing $F$.

On overlapping characteristic charts the vertical fields agree on
$\Sigma$, since $a_F=F/\alpha$. Both $a_FT$ and $-F(\alpha)e_{2n}$
have normalized normal speed $F$ on their regular domains, so
\[
 \varphi_{X_F}
 =\chi(\cos u)\frac{F(\alpha)}{\alpha}\alpha
  +(1-\chi(\cos u))F(\alpha)=F(\alpha).
\]
Since $\chi(\cos u)=1$ near the characteristic set, the displayed
surface field equals $a_FT$ there and is $C^1$ across that set.
Extend it locally in $C^2$ hypersurface charts, using vertical
extensions near characteristic points. A finite ambient partition
of unity, with noncharacteristic members supported away from the
characteristic set, glues these extensions to a global $C^1$ field
with the stated surface restriction.
\end{proof}

Fix such an $X_F$. Under the standing CMC assumptions,
any two choices agree on $\Sigma$ near the
characteristic set and differ by a tangential field supported in
$\Sigma^*$. Locality and the $C^1$ regular pairing
\cite{JacobiSpectrum}*{Proposition~2.9}
make every Hessian pairing with this difference vanish.
Their volume fluxes also agree.
Thus the volume flux and Hessian pairings are independent of the
chosen realization in this class.

The same graph description shows how the factors $\cos u$ in the centre
differential control the growth of $\alpha$ and $\Shape$ near
the characteristic set:

\Needspace{4\baselineskip}
\begin{lemma}[Characteristic graph bounds]\label{lem:characteristic-bounds}
In a fixed compact characteristic graph, put $G=\nabla f+Jz$,
$g=|G|$ and $\sigma=\operatorname{sgn}\langle N,T\rangle$.
On its regular part $\alpha=\sigma/g$, $\cos u=g/\sqrt{1+g^2}$, and
\begin{equation}\label{eq:characteristic-bounds}
 \nabla_b u=-\frac{\sigma}{1+g^2}\nabla_b g,\qquad
 |\nabla_b u|+|\cos u\,\Shape|+|\cos u\,\ShapeError |+|\cos^2 u\,q_\Sigma|=O(1).
\end{equation}
In particular, $\AngleError$, $\cos u\,\ShapeError$ and $D_Hz_0$ are bounded,
so the centre and angle energies are integrable.
At $H_\Sigma=2n$, the Hessian density and the divergence fields satisfy
\begin{equation}\label{eq:characteristic-multiplier-bounds}
 \begin{gathered}
 |LF|+|\cos u\,F|+|F'(\alpha)|=O(1),\qquad
 |F\,LF|\,d\sigma_H\le C\,dz,\\
 |u\nabla_b u|+|ue_n|=O(1),\qquad
 |F'(\alpha)\ShapeError e_n|=O(g^{-1}).
 \end{gathered}
\end{equation}
The first line controls the Hessian pairing; the second gives the
field growth needed to remove divergences across the characteristic set.
\end{lemma}
\begin{proof}
Choose $M\ge1$ bounding both $g$ and $|D_zG|$ on the fixed compact
chart; below, $C$ depends only on $M$ and $n$.
On $\{g>0\}$, writing $G$ as a column vector gives
\[
 D_zg=\frac{G^{\mathsf T}D_zG}{g},\qquad
 D_z(G/g)=\frac{D_zG}{g}
 -\frac{G(G^{\mathsf T}D_zG)}{g^3}.
\]
Since $|G|=g$, these formulas give, by Cauchy--Schwarz,
\[
 |\nabla_b g|\le|D_zg|\le M,\qquad
 |D_z(G/g)|\le\frac{M}{g}+\frac{g^2M}{g^3}
 =\frac{2M}{g}.
\]
Here $e_{2n}=\sigma G/g$, and the fixed horizontal frame is parallel.
The definition $\Shape X=-\nabla_Xe_{2n}+\alpha J'X$, with
$|\alpha|=g^{-1}$ and $|J'|=\sqrt{2n-2}$, therefore gives
\[
 g|\Shape|\le2M+\sqrt{2n-2},\qquad
 g|\ShapeError|\le g|\Shape|+g|P|
 \le2M+\sqrt{2n-2}+M\sqrt{2n+2}.
\]
Since $\frac{\cos u}{g}=(1+g^2)^{-1/2}\le1$, both $|\cos u\,\Shape|$ and
$|\cos u\,\ShapeError|$ are bounded by $C$.

The chain rule and $|e_ng|\le|\nabla_b g|\le M$ give
\[
 \begin{gathered}
 \nabla_bu=\cos^2 u\,\nabla_b\alpha
 =-\frac{\sigma}{1+g^2}\nabla_b g,\qquad
 |\nabla_bu|\le M,\\
 |e_n\alpha|=g^{-2}|e_ng|\le\frac{M}{g^2},\qquad
 \cos^2 u\,|e_n\alpha|\le\frac{M}{1+g^2},\qquad
 \cos^2 u\,\alpha^2=\frac1{1+g^2}.
 \end{gathered}
\]
Thus the potential \eqref{eq:potential-symmetric-shape} and the
centre differential \eqref{rig:eq:centre-defect-differential} satisfy
\[
 \begin{gathered}
 |\cos^2 u\,q_\Sigma|\le|\cos u\,\Shape|^2+4\cos^2 u\,|e_n\alpha|
       +2(n+1)\cos^2 u\,\alpha^2\le C,\\
 |\AngleError|\le|\nabla_bu|+1\le M+1,\qquad
 |D_Hz_0|\le|\cos u\,\ShapeError|+|\AngleError|\le C.
 \end{gathered}
\]
This proves \eqref{eq:characteristic-bounds}; the centre and angle
energies are integrable since $d\sigma_H=g\,dz$ and $g\le M$.

At $H_\Sigma=2n$, \eqref{rig:eq:LF-direct} gives
$|LF|\le|\cos^2 u\,q_\Sigma|+2|\nabla_bu|^2\le C$.
Using $|u|\le\pi/2$, $|\sin u|\le1$ and
$|\alpha|\cos^2 u=g/(1+g^2)\le1/2$, we also have
\[
 |\cos u\,F|=|\cos u+u\sin u|\le1+\frac\pi2,\qquad
 |F'(\alpha)|=|u+\alpha \cos^2 u|\le\frac\pi2+\frac12.
\]
Since $F=1+u\alpha$ and $|\alpha|=g^{-1}$,
\[
 g|F|\le g+\frac\pi2\le M+\frac\pi2,\qquad
 |F\,LF|\,d\sigma_H=g|F||LF|\,dz\le C\,dz.
\]
Finally, the divergence fields satisfy
\[
 |u\nabla_bu|+|ue_n|\le\frac\pi2(M+1),\qquad
 |F'(\alpha)\ShapeError e_n|
 \le\frac{\pi+1}{2}|\ShapeError|\le\frac{C}{g},
\]
which proves \eqref{eq:characteristic-multiplier-bounds}.
\end{proof}

\subsection{The volume-constrained test field}\label{subsec:volume-correction}
Choose the multiple of $D$ so that $W$ satisfies the linearized volume constraint:

\begin{equation}\label{eq:constrained-field-construction}
 \kappa=\frac{\delta V[X_F]}{\delta V[D]},\qquad
 W=X_F-\kappa D,\qquad \delta V[W]=0.
\end{equation}
Equation~\eqref{eq:volume-hessian-orthogonality} extends by $C^1$ continuity
to $W$, so $\delta V[W]=0$ gives $\mathcal B(D,W)=0$.
To apply stability to $W$, let $X$ be any $C^1$ field with $\delta V[X]=0$
and choose smooth $X_j\to X$ in $C^1$ near $\Sigma$.
If $\Sigma$ is stable, Lemma~\ref{lem:stationarity-cmc} gives
\begin{equation}\label{rig:eq:C1-stability}
 W_j=X_j-\frac{\delta V[X_j]}{\delta V[D]}D
 \longrightarrow X,
 \qquad \mathcal B(X,X)=\lim_j\mathcal B(W_j,W_j)\ge0.
\end{equation}
Since $X_j\to X$ in $C^1$ and $\delta V[X_j]\to\delta V[X]=0$
by $C^0$ continuity of volume flux, the definition of $W_j$ gives
$W_j\to X$ in $C^1$;
the $C^1$ continuity of the Hessian in
\cite{JacobiSpectrum}*{Proposition~2.6} gives the limit.
Taking $X=W$ yields $\mathcal B(W,W)\ge0$.

Dilation orthogonality also gives
\begin{equation}\label{eq:canonical-completed-square}
 \mathcal B(X_F,X_F)
 =\mathcal B(W,W)-2n\delta V[D]\,\kappa^2.
\end{equation}
Changing the dilation centre also leaves the Hessian value of the
resulting volume-constrained field unchanged.
For $a=(a_z,a_t)$, let $D_a$ be the infinitesimal generator of
$L_a\circ\delta_{e^\tau}\circ L_a^{-1}$, where $L_a$ is left translation.
The group law gives
\[
 D_a-D=-a_z\cdot\partial_z+
 \bigl(\langle Ja_z,z\rangle-2a_t\bigr)\partial_t.
\]
This is the generator of left translation by $(-\tau a_z,-2\tau a_t)$.
Translations preserve perimeter and volume, so $\delta V[D_a-D]=0$;
stationarity gives $\mathcal B(D_a-D,X)=0$ for $C^2$ fields, and
the continuity in \cite{JacobiSpectrum}*{Proposition~2.6} extends this to $C^1$ fields.
Consequently $D_a$ has the same flux as $D$, $\kappa$ is unchanged, and
$\mathcal B(X_F-\kappa D_a,X_F-\kappa D_a)=\mathcal B(W,W)$.

\Needspace{16\baselineskip}
\section{The centre identity and CMC rigidity}\label{sec:positive-square}
\label{rig:sec:constrained-balance}
The volume correction subtracts a scalar square from the angle energy,
but Cauchy--Schwarz leaves a positive coefficient on that energy. Hence
$\mathcal B(W,W)\ge0$ forces both defects to vanish.

\PansuResultBookmark{2}{rig:prop:quantitative-instability}{Theorem}
\begin{theorem}[Constrained Hessian identity]
\label{rig:prop:quantitative-instability}
Let $\Omega\subset\Hn$, $n\ge2$, be a bounded nonempty
open set with $C^2$ boundary $\Sigma=\partial\Omega$ and outward
horizontal mean curvature $H_\Sigma=2n$ on $\Sigma^*$.
For the field $X_F$ of Lemma~\ref{rig:lem:characteristic-admissibility},
put $\kappa=\delta V[X_F]/\delta V[D]$ and $W=X_F-\kappa D$, and set
\begin{equation}\label{eq:area-angle-mean}
 A=\PH(\Sigma),\qquad
 \beta=\int_\Sigma\langle \AngleError,e_n\rangle\,d\sigma_H.
\end{equation}
Since $\langle\AngleError,e_n\rangle=e_nu+1$, $\beta$ is the
integrated signed deviation from the Pansu angle rate $-1$.
Then $\delta V[W]=0$, all integrals below are finite, and
\begin{equation}\label{rig:eq:constrained-integral-balance}
 -\mathcal B(W,W)
 =\int_\Sigma(|D_Hz_0|^2+2|\AngleError|^2)\,d\sigma_H
 -\frac{\beta^2}{(2n+1)A}.
\end{equation}
Since $\beta^2\le A\int_\Sigma|\AngleError|^2\,d\sigma_H$ by Cauchy--Schwarz,
\begin{equation}\label{eq:canonical-defect-lower-bound}
 -\mathcal B(W,W)
 \ge\int_\Sigma\left(|D_Hz_0|^2+\frac{4n+1}{2n+1}|\AngleError|^2\right)\,d\sigma_H.
\end{equation}
If $\mathcal B(W,W)\ge0$, then on $\Sigma^*$
\begin{equation}\label{rig:eq:pansu-system}
 D_Hz_0=0,\qquad \AngleError=0,\qquad \ShapeError =0.
\end{equation}
\end{theorem}
\subsection{The local identity from the centre differential}
\label{subsec:local-centre-identity}
Squaring the centre differential produces a mixed shape and angle
term. We combine it with the Jacobi potential and a CMC conservation
law to obtain a divergence identity.
For a horizontal tangent field $Z$, the weighted divergence is
\begin{equation}\label{eq:weighted-horizontal-divergence}
 \operatorname{div}_HZ
 =\operatorname{div}Z+2\alpha\langle Z,e_n\rangle,
\end{equation}
where $\operatorname{div}$ is the trace of the induced horizontal
connection.
The characteristic direction satisfies
\begin{equation}\label{rig:eq:characteristic-direction-derivative}
 (\nabla_Xe_n)^{\top_H}=J'\Shape X
   +\alpha(X-\langle X,e_n\rangle e_n),\qquad
 \operatorname{div}_He_n=2n\alpha.
\end{equation}
Indeed, $e_n=-Je_{2n}$, $\nabla J=0$ and
$\nabla_Xe_{2n}=-(\Shape-\alpha J')X$ give
\[
 \begin{aligned}
 (\nabla_Xe_n)^{\top_H}
 &=(-J\nabla_Xe_{2n})^{\top_H}
   =(J\Shape X-\alpha JJ'X)^{\top_H}\\
 &=J'\Shape X+\alpha(X-\langle X,e_n\rangle e_n).
 \end{aligned}
\]
Here $(JY)^{\top_H}=J'Y$ and
$JJ'X=-X+\langle X,e_n\rangle e_n$ for horizontal tangent vectors.
Since $J'$ is skew-adjoint and $\Shape$ is self-adjoint,
\[
 \operatorname{tr}(J'\Shape)
 =\operatorname{tr}\bigl((J'\Shape)^{\mathsf T}\bigr)
 =-\operatorname{tr}(\Shape J')=-\operatorname{tr}(J'\Shape)=0.
\]
For a horizontal tangent orthonormal frame $(E_a)_{a=1}^{2n-1}$,
\eqref{eq:weighted-horizontal-divergence} therefore gives
\[
 \begin{aligned}
 \operatorname{div}_He_n
 &=\sum_a\langle(\nabla_{E_a}e_n)^{\top_H},E_a\rangle+2\alpha\\
 &=\operatorname{tr}(J'\Shape)
   +\alpha\sum_a(1-\langle E_a,e_n\rangle^2)+2\alpha\\
 &=0+\alpha\bigl((2n-1)-1\bigr)+2\alpha=2n\alpha.
 \end{aligned}
\]

\begin{lemma}[Shape conservation at constant mean curvature]
\label{lem:cmc-shape-conservation}
On a $C^2$ regular patch with $H_\Sigma=2n$, one has weakly,
including across $\alpha=0$,
\begin{equation}\label{rig:eq:codazzi-conservation}
 \operatorname{div}_H(\ShapeError e_n)=0.
\end{equation}
\end{lemma}
\begin{proof}
Appendix~\ref{app:codazzi-graph} derives the identity from Codazzi
and proves its weak extension at $C^2$ regularity.
\end{proof}
The model height supplies the multiplier: $F'(\alpha)=2t_{\mathrm P}(u)$
and the lift relation give $\nabla_bF'(\alpha)=2\cos^2 u\,\nabla_bu$.
This matches the $\nabla_bu$ part of the mixed term in
\eqref{eq:centre-differential-square}; its $e_n$ part cancels against
the Jacobi potential. The conservation law turns the remaining
mixed term into a divergence.

\Needspace{8\baselineskip}
\begin{proposition}[Local centre identity]
\label{rig:lem:structure}
Under the standing CMC assumptions, one has weakly on $\Sigma^*$
\begin{equation}\label{eq:local-defect-identity}
 \begin{aligned}
 F\,LF={}&-|D_Hz_0|^2-2|\AngleError|^2+2\langle \AngleError,e_n\rangle-(2n+1)\\
 &+\operatorname{div}_H
       \bigl(F'(\alpha)\ShapeError e_n+u\nabla_b u\bigr).
 \end{aligned}
\end{equation}
\end{proposition}
\begin{proof}
Since $PX=X+\langle X,e_n\rangle e_n$ and
$\operatorname{tr}\ShapeError=H_\Sigma-2n=0$, an orthonormal frame gives
\[
 \begin{aligned}
 \langle P,\ShapeError\rangle
 &=\sum_a\langle PE_a,\ShapeError E_a\rangle\\
 &=\sum_a\langle E_a,\ShapeError E_a\rangle
   +\sum_a\langle E_a,e_n\rangle
     \langle e_n,\ShapeError E_a\rangle\\
 &=\operatorname{tr}\ShapeError+\langle\ShapeError e_n,e_n\rangle
   =\langle\ShapeError e_n,e_n\rangle.
 \end{aligned}
\]
Moreover, $\nabla_bu=\cos^2 u\,\nabla_b\alpha$ and
$\AngleError=\nabla_bu+e_n$ imply
\[
 \cos^2 u\,e_n\alpha=e_nu=\langle\AngleError,e_n\rangle-1,
 \qquad \cos^2 u\,\alpha^2=\sin^2 u.
\]
Using $\Shape=P+\ShapeError$, $|P|^2=2n+2$ and $\cos^2 u+\sin^2 u=1$,
we can now expand the potential:
\begin{equation}\label{rig:eq:jacobi-potential}
 \begin{aligned}
 \cos^2 u\,q_\Sigma
 &=\cos^2 u\,\bigl(|P|^2+2\langle P,\ShapeError\rangle+|\ShapeError|^2\bigr)
   +4\cos^2 u\,e_n\alpha+2(n+1)\cos^2 u\,\alpha^2\\
 &=|\cos u\,\ShapeError|^2+2\cos^2 u\,\langle\ShapeError e_n,e_n\rangle
   +2(n+1)\cos^2 u\\
 &\quad+4\bigl(\langle\AngleError,e_n\rangle-1\bigr)+2(n+1)\sin^2 u\\
 &=2(n-1)+|\cos u\,\ShapeError|^2+2\cos^2 u\,\langle\ShapeError e_n,e_n\rangle
   +4\langle\AngleError,e_n\rangle.
 \end{aligned}
\end{equation}
Substitution into \eqref{rig:eq:LF-direct}, using
$|\nabla_bu|^2=|\AngleError|^2-2\langle\AngleError,e_n\rangle+1$,
then gives
\begin{equation}\label{eq:local-centre-codazzi}
 \begin{aligned}
 LF&=-\cos^2 u\,q_\Sigma-2|\nabla_bu|^2\\
 &=-2(n-1)-\cos^2 u\,|\ShapeError|^2
   -2\cos^2 u\,\langle\ShapeError e_n,e_n\rangle-4\langle\AngleError,e_n\rangle\\
 &\quad-2\bigl(|\AngleError|^2-2\langle\AngleError,e_n\rangle+1\bigr)\\
 &=-2n-\cos^2 u\,|\ShapeError|^2-2|\AngleError|^2
   -2\cos^2 u\,\langle\ShapeError e_n,e_n\rangle.
 \end{aligned}
\end{equation}
Since $F'(\alpha)\in C^1(\Sigma^*)$, testing
Lemma~\ref{lem:cmc-shape-conservation} with $F'(\alpha)\psi$,
$\psi\in C_c^1(\Sigma^*)$, gives
\[
 \begin{aligned}
 0&=\int_\Sigma\langle\ShapeError e_n,
     \nabla_b(F'\psi)\rangle\,d\sigma_H\\
 &=\int_\Sigma F'\langle\ShapeError e_n,\nabla_b\psi\rangle\,d\sigma_H
   +2\int_\Sigma \cos^2 u\,\psi
     \langle\nabla_bu,\ShapeError e_n\rangle\,d\sigma_H,
 \end{aligned}
\]
where $\nabla_bF'=F''\nabla_b\alpha=2\cos^2 u\,\nabla_bu$.
Consequently, in the weak sense,
\[
 \begin{aligned}
 \operatorname{div}_H\bigl(F'(\alpha)\ShapeError e_n\bigr)
 &=2\cos^2 u\,\langle\nabla_bu,\ShapeError e_n\rangle\\
 &=2\cos^2 u\,\langle\AngleError,\ShapeError e_n\rangle
   -2\cos^2 u\,\langle\ShapeError e_n,e_n\rangle.
 \end{aligned}
\]
The centre square \eqref{eq:centre-differential-square} now identifies
the first bracket in the following calculation:
\[
 \begin{aligned}
 LF&=-\bigl(\cos^2 u\,|\ShapeError|^2+|\AngleError|^2
        +2\cos^2 u\,\langle\AngleError,\ShapeError e_n\rangle\bigr)
      -|\AngleError|^2-2n\\
 &\quad+\bigl(2\cos^2 u\,\langle\AngleError,\ShapeError e_n\rangle
       -2\cos^2 u\,\langle\ShapeError e_n,e_n\rangle\bigr)\\
 &=-|D_Hz_0|^2-|\AngleError|^2-2n
    +\operatorname{div}_H\bigl(F'(\alpha)\ShapeError e_n\bigr).
 \end{aligned}
\]

To pass from $LF$ to $F\,LF$, use $F=1+u\alpha$.
For $\psi\in C_c^1(\Sigma^*)$, testing $L\alpha=0$ with
$\cos^2 u\,\psi$ and using $\nabla_b(\cos^2 u)=-2\alpha \cos^4 u\,\nabla_b\alpha$ gives
\[
 \begin{aligned}
 0&=\int_\Sigma\bigl(\langle\nabla_b\alpha,
     \nabla_b(\cos^2 u\,\psi)\rangle-q_\Sigma\alpha \cos^2 u\,\psi\bigr)\,d\sigma_H\\
 &=\int_\Sigma\langle\nabla_bu,\nabla_b\psi\rangle\,d\sigma_H
   -\int_\Sigma\bigl(2\alpha|\nabla_bu|^2
     +q_\Sigma\alpha \cos^2 u\bigr)\psi\,d\sigma_H.
 \end{aligned}
\]
Thus the weak Laplacian of $u=\arctan\alpha$ is
\[
\begin{aligned}
 \Delta_b  u
 &=-\alpha \cos^2 u\,q_\Sigma-2\alpha|\nabla_bu|^2\\
 &=\alpha\bigl(-\cos^2 u\,q_\Sigma-2|\nabla_bu|^2\bigr)=\alpha LF.
\end{aligned}
\]
This identity absorbs the extra term $u\alpha LF$ into a divergence:
\[
 \begin{aligned}
 \operatorname{div}_H(u\nabla_bu)
 &=\langle\nabla_bu,\nabla_bu\rangle+u\Delta_bu\\
 &=|\nabla_bu|^2+u\alpha LF.
 \end{aligned}
\]
Since $F=1+u\alpha$, it follows that
\begin{equation}\label{rig:eq:FLF-divergence}
 F\,LF=LF+\operatorname{div}_H(u\nabla_b u)-|\nabla_b u|^2.
\end{equation}
Inserting the centre expression for $LF$ makes the remaining
scalar terms explicit:
\[
 \begin{aligned}
 F\,LF
 &=-|D_Hz_0|^2-|\AngleError|^2-2n
   -\bigl(|\AngleError|^2-2\langle\AngleError,e_n\rangle+1\bigr)\\
 &\quad+\operatorname{div}_H
   \bigl(F'(\alpha)\ShapeError e_n+u\nabla_bu\bigr)\\
 &=-|D_Hz_0|^2-2|\AngleError|^2+2\langle\AngleError,e_n\rangle-(2n+1)\\
 &\quad+\operatorname{div}_H
   \bigl(F'(\alpha)\ShapeError e_n+u\nabla_bu\bigr),
 \end{aligned}
\]
which is \eqref{eq:local-defect-identity}.
\end{proof}

\subsection{Global pairing across the characteristic set}
\label{subsec:global-centre-pairing}
Global integration requires integrable scalar terms and zero flux
across the characteristic set. Lemma~\ref{lem:characteristic-bounds}
gives
\[
 |D_Hz_0|^2+|\AngleError|^2+|\nabla_bu|^2+|\cos^2 u\,q_\Sigma|=O(1),
 \qquad F=O(g^{-1})
\]
in each characteristic chart. Since $d\sigma_H=g\,dz$, these
quantities are integrable. The same lemma makes $F\,LF$ integrable
and bounds $F'(\alpha)\ShapeError e_n+u\nabla_bu$ by $Cg^{-1}$.

To determine $\kappa$, we first express the first variation of volume
$\delta V[X_F]$ in terms of $A$ and $\beta$.
\begin{proposition}
\label{prop:finite-defect-balance}
With $A$ and $\beta$ as in \eqref{eq:area-angle-mean}, one has
\begin{equation}\label{rig:eq:normal-angle-balance}
 2n\delta V[X_F]=(2n+1)A-\beta.
\end{equation}
\end{proposition}
\begin{proof}
The bounded field $ue_n$ has integrable divergence
$e_nu+2n\alpha u$. The divergence extension in
Lemma~\ref{app:lem:no-divergence-charge} gives zero total divergence.
Since $e_nu=\langle\AngleError,e_n\rangle-1$ and $\alpha u=F-1$,
this is
$0=\beta-A+2n\bigl(\delta V[X_F]-A\bigr)$,
which proves the formula.
\end{proof}
We next express the Hessian through the integrable density $F\,LF$.

\begin{lemma}[Global pairing for vertical characteristic representatives]
\label{rig:lem:characteristic-pairing}
Let $\Sigma=\partial\Omega$ be compact, oriented and $C^2$, with
$H_\Sigma=2n$ on $\Sigma^*$, where $n\ge2$. Let $X$ be a $C^1$ surface
field with normalized speed $\varphi$ on $\Sigma^*$. Suppose $L\varphi$
has a locally integrable weak representative there and
$\varphi L\varphi\in L^1(\Sigma,d\sigma_H)$. In every characteristic
vertical graph $t=f(z)$, put $G=\nabla f+Jz$ and $g=|G|$. Suppose that
near its characteristic set $X=aT$, where $a$ extends in $C^1$ across
$\{g=0\}$ and, on each compact subchart,
\[
 |a|\le C,\qquad |\nabla_ba|\le Cg^2.
\]
Then the global Hessian satisfies
\begin{equation}\label{rig:eq:vertical-global-pairing}
 \mathcal B(X,X)=\int_\Sigma\varphi L\varphi\,d\sigma_H.
\end{equation}
\end{lemma}
\begin{proof}
Choose a finite partition $\sum_i\eta_i=1$ near $\Sigma$ subordinate
to regular charts and characteristic graphs where $X=aT$. Locality
and bilinearity of the $C^1$ Hessian
\cite{JacobiSpectrum}*{Proposition~2.6} allow us to partition its first
argument. On each chart, a fixed cutoff equal to one near
$\operatorname{supp}\eta_i\cap\Sigma$ localizes the second argument.
On regular charts, the local pairing
\cite{JacobiSpectrum}*{Proposition~2.9} and weak integration by parts
give $\mathcal B(\eta_iX,X)=\int\eta_i\varphi L\varphi\,d\sigma_H$.

On a characteristic graph, write
$\Pi_G=I-g^{-2}G\otimes G$ on $\{g>0\}$ and $\alpha=\sigma/g$.
For compactly supported vertical heights $a_1,a_2$, the norm
Hessian formula \cite{JacobiSpectrum}*{Section~2.1} applied to
the graph density $|\nabla f+Jz|$ gives
\begin{equation}\label{rig:eq:vertical-multiplier}
 \mathcal B(a_1T,a_2T)
 =\int\frac{\langle\Pi_G\nabla_za_1,\Pi_G\nabla_za_2\rangle}{g}\,dz
 =\int\alpha^2\langle\nabla_ba_1,\nabla_ba_2\rangle\,d\sigma_H.
\end{equation}
The contribution of a vertical graph patch to enclosed volume is
affine in the graph height, so the volume Hessian vanishes for vertical
height fields. The rank estimate
\cite{JacobiSpectrum}*{Corollary~2.2} makes $g^{-1}$ integrable for
$n\ge2$; continuity of the Hessian extends the formula to
$C^1$ heights.

Central translation gives $L\alpha=0$ weakly on $\Sigma^*$ by
Proposition~\ref{lem:central-jacobi}. Testing it with $v\psi$ yields
the ground-state identity, wherever $\alpha\ne0$,
\begin{equation}\label{eq:vertical-conjugation}
 L(\alpha v)=-\alpha^{-1}\operatorname{div}_H(\alpha^2\nabla_bv).
\end{equation}
Since $\varphi=\alpha a$, applying the product rule to this identity
with $v=a$ gives on the regular part of the graph
\begin{equation}\label{rig:eq:FLF-localization}
 \alpha^2\langle\nabla_b(\eta_i a),\nabla_ba\rangle
 =\eta_i\varphi L\varphi+
   \operatorname{div}_H(\eta_i\varphi\alpha\nabla_ba).
\end{equation}
At $C^1$ regularity, testing the weak equation $L\alpha=0$ on compact
subsets of $\Sigma^*$ establishes \eqref{rig:eq:FLF-localization} in
distributions. Our assumptions give
$\eta_i\varphi\alpha\nabla_ba
 =\eta_i\alpha^2a\nabla_ba=O(1)$ and
$\alpha^2\langle\nabla_b(\eta_i a),\nabla_ba\rangle=O(1)$.
Thus the divergence density in \eqref{rig:eq:FLF-localization} is
integrable. Lemma~\ref{app:lem:no-divergence-charge} extends the
identity across $\{g=0\}$ with zero total divergence. Equation
\eqref{rig:eq:vertical-multiplier} now gives
$\mathcal B(\eta_iX,X)=\int\eta_i\varphi L\varphi\,d\sigma_H$
on the characteristic graph as well. Summing over the partition
proves \eqref{rig:eq:vertical-global-pairing}.
\end{proof}

\begin{proposition}[Hessian identity for $X_F$]
\label{rig:lem:hessian-localization}\label{rig:lem:exact-admissibility}
Under the standing CMC assumptions,
\begin{equation}\label{rig:eq:FLF-hessian}
 \mathcal B(X_F,X_F)=\int_\Sigma F\,LF\,d\sigma_H.
\end{equation}
\end{proposition}
\begin{proof}
Lemma~\ref{rig:lem:characteristic-admissibility} gives
$X_F=a_FT$ near the characteristic set. Its height formula
\eqref{rig:eq:canonical-heights} gives $a_F=O(1)$ and
$\nabla_ba_F=O(g^2)$. Lemma~\ref{lem:characteristic-bounds}
gives $F\,LF\in L^1(d\sigma_H)$. Apply
Lemma~\ref{rig:lem:characteristic-pairing} with $\varphi=F$.
\end{proof}

All scalar terms in \eqref{eq:local-defect-identity} are integrable
by the bounds above, so its divergence density is integrable.
Together with the $O(g^{-1})$ field bound, this verifies both
hypotheses of Lemma~\ref{app:lem:no-divergence-charge}; the identity
extends across the characteristic set with zero total divergence.
Integrating \eqref{eq:local-defect-identity} and using
\eqref{rig:eq:FLF-hessian} therefore gives
\begin{equation}\label{eq:integrated-defect-pairing}
 \mathcal B(X_F,X_F)
 =-\int_\Sigma(|D_Hz_0|^2+2|\AngleError|^2)\,d\sigma_H
   +2\beta-(2n+1)A.
\end{equation}

\subsection{The constrained Hessian}
\label{subsec:centre-volume-correction}
\begin{proof}[Proof of \cref{rig:prop:quantitative-instability}]
By definition,
$\kappa=\delta V[X_F]/\delta V[D]$ and $\delta V[W]=0$.
The completed square
\eqref{eq:canonical-completed-square} gives
\[
 \mathcal B(W,W)=\mathcal B(X_F,X_F)
 +\frac{2n\bigl(\delta V[X_F]\bigr)^2}{\delta V[D]}.
\]
Since $2n\delta V[D]=(2n+1)A$
by Lemma~\ref{lem:stationarity-cmc} at $h=2n$,
equation~\eqref{rig:eq:normal-angle-balance} gives
\[
 \frac{2n\bigl(\delta V[X_F]\bigr)^2}{\delta V[D]}
 =\frac{\bigl((2n+1)A-\beta\bigr)^2}{(2n+1)A}
 =(2n+1)A-2\beta+\frac{\beta^2}{(2n+1)A}.
\]
Substitution into \eqref{eq:integrated-defect-pairing} cancels
both the constant and linear terms and proves
\eqref{rig:eq:constrained-integral-balance}.
The scalar correction involves the perimeter average $\beta/A$
of $e_nu+1$. The Cauchy--Schwarz estimate
in the statement therefore gives \eqref{eq:canonical-defect-lower-bound}.
If $\mathcal B(W,W)\ge0$, continuity on $\Sigma^*$ gives
$D_Hz_0=\AngleError=0$. Then
\eqref{eq:centre-differential-square}, with $\cos u>0$ there, gives
$\ShapeError=0$.
\end{proof}

On the unit Pansu sphere centred at the origin,
$\ShapeError=\AngleError=D_Hz_0=0$ and $\beta=0$.
Equation~\eqref{eq:dilation-support-speed} gives
$\varphi_{X_F}=\varphi_D$, so $\kappa=1$.
Since \eqref{rig:eq:LF-direct} gives $LF=-2n$, the identities reduce to
\[
 \mathcal B(X_F,X_F)=-2n\delta V[X_F]=-(2n+1)A,
 \qquad \mathcal B(W,W)=0.
\]
This checks the Hessian sign and the dilation correction on the sphere.

\subsection{Rigidity from the centre identity}\label{rig:sec:classification}
The centre identity determines the geometry of each component.
The number of components is a separate variational question: the
single field $W$ can vanish in the Hessian on a union of spheres,
but transferring volume between components gives negative directions.
To make this distinction precise, return temporarily to any constant
mean curvature $h>0$. The compact boundary has
finitely many components $\Sigma_1,\ldots,\Sigma_m$.
Choose smooth $D_i$ equal to $D$ near $\Sigma_i$
and zero near all other components, with mutually disjoint supports.
Put $v_i=\delta V[D_i]$. The flow of $D_i$ dilates $\Sigma_i$
and fixes the other components. Stationarity and homogeneity give
\[
 hv_i=\delta\PH[D_i]
 =\left.\frac{d}{d\tau}\right|_0 e^{(2n+1)\tau}\PH(\Sigma_i)
 =(2n+1)\PH(\Sigma_i)>0.
\]
For $i\ne j$, disjoint supports give $\mathcal B_h(D_i,D_j)=0$.
Since $D_i-D$ vanishes near $\Sigma_i$, locality and
\eqref{eq:volume-hessian-orthogonality} give
\[
 \mathcal B_h(D_i,D_i)=\mathcal B_h(D,D_i)=-h\delta V[D_i]=-hv_i.
\]
The $m$ fields span a negative-definite space, whose subspace
$\sum_i a_iv_i=0$ has dimension $m-1$ and zero volume flux. Hence
\begin{equation}\label{eq:component-index}
 \operatorname{ind}_{\mathrm{vol}}(\mathcal B_h)\ge m-1.
\end{equation}
If $m\ge2$, choose a nonzero smooth field in this zero-flux subspace.
Lemma~\ref{lem:stationarity-cmc} realizes it by an exactly
volume-preserving variation with negative second perimeter derivative.
Thus volume-preserving stability makes the boundary connected. The exact
conclusion supplied by $W$ alone, at the normalized scale, is the
following.

\begin{corollary}[Single-field criterion]\label{thm:single-field-classification}
Let $\Omega\subset\Hn$, $n\ge2$, be a bounded nonempty
open set with $C^2$ boundary $\Sigma$ and outward
$H_\Sigma=2n$ on $\Sigma^*$. Then $\mathcal B(W,W)\ge0$ for the
field in \eqref{eq:constrained-field-construction} if and only if
$\Sigma$ is a finite union of unit Pansu spheres enclosing balls
with pairwise disjoint closures. If $\Sigma$ is connected, this
is exactly one left translate of the unit Pansu sphere.
\end{corollary}
\PansuResultBookmark{3}{thm:single-field-classification}{Corollary}
\begin{proof}
If $\mathcal B(W,W)\ge0$, Theorem~\ref{rig:prop:quantitative-instability}
gives $D_Hz_0=\AngleError=0$. The compact embedded boundary has finitely
many components; Proposition~\ref{prop:zero-defect-reconstruction}
identifies each with an outward unit Pansu sphere. Distinct enclosed
balls have disjoint closures: disjoint embedded sphere boundaries
enclose either disjoint or nested domains, and strict nesting of
two unit Pansu balls contradicts their equal volume. Conversely,
these spheres have $D_Hz_0=\AngleError=\beta=0$, so
\eqref{rig:eq:constrained-integral-balance} gives
$\mathcal B(W,W)=0$.
\end{proof}

If $\Sigma$ in this corollary is connected and is not a Pansu sphere, $\mathcal B(W,W)<0$.
The smooth approximants in \eqref{rig:eq:C1-stability} satisfy the
linearized volume constraint and retain strict negativity by $C^1$ continuity. The volume-preserving adjustment in
Lemma~\ref{lem:stationarity-cmc} realizes them by smooth
volume-preserving destabilizing variations.

\begin{proof}[Proof of \textup{(ii)} $\Rightarrow$ \textup{(i)} in \cref{rig:thm:stable-cmc}]
Suppose the boundary is stable.
Lemma~\ref{lem:stationarity-cmc} gives $H_\Sigma=h>0$ and the
Hessian inequality under the linearized volume constraint; \eqref{eq:component-index} gives
connectedness.
Dilate by $h/(2n)$ as in
Section~\ref{subsec:normalization} to work at unit scale.
The approximants in \eqref{rig:eq:C1-stability} satisfy the linearized
volume constraint and give $\mathcal B(W,W)\ge0$ from $C^2$ ambient stability.
Corollary~\ref{thm:single-field-classification} identifies the
normalized boundary as a unit Pansu sphere. Undoing the dilation
gives radius $2n/h$.
\end{proof}

\Needspace{9\baselineskip}
\section{Stability of Pansu spheres}\label{sec:model-stability}
\label{subsec:form-domain}
The converse follows by transferring the model index bound to ambient
fields through the global energy representation.

\begin{proof}[Remaining implications in \cref{rig:thm:stable-cmc}]
Lemma~\ref{lem:stationarity-cmc} gives stationarity of Pansu spheres.
On the unit Pansu sphere,
with the outward normal and perimeter measure of
Section~\ref{sec:regular}, the pole estimates in
\cite{JacobiSpectrum}*{Proposition~5.1} place the normalized speed
of every real $C^1$ surface field in the energy domain of
\cite{JacobiSpectrum}*{Definition~4.1}.
On that energy domain, the global
second-variation result~\cite{JacobiSpectrum}*{Corollary~1.2} gives
\[
 \mathcal B(W_1,W_2)
 =Q_{\mathcal S}(\varphi_{W_1},\varphi_{W_2}),\qquad
 \varphi_W=\frac{\langle W,N\rangle}{|N_b|},
\]
for all such fields, including $C^2$ ambient fields that move the poles.
The complete spectrum in \cite{JacobiSpectrum}*{Theorem~1.3}
shows that $Q_{\mathcal S}$ has index one on the real energy space.
Now let $E$ be a finite-dimensional $\mathcal B$-negative subspace of
ambient fields. The speed map $W\mapsto\varphi_W$ is injective on $E$:
if $\varphi_W=0$, then $\mathcal B(W,W)=Q_{\mathcal S}(0,0)=0$, forcing $W=0$.
Its image is therefore a $Q_{\mathcal S}$-negative subspace of the same
dimension, so $\dim E\le1$. Taking the supremum gives
$\operatorname{ind}(\mathcal B)\le1$.
Dilation supplies a negative direction, so the index is one.
Equation~\eqref{eq:ambient-volume-index} gives volume-preserving
stability, proving \textup{(i)} $\Rightarrow$ \textup{(ii)},
as in \cite{JacobiSpectrum}*{Corollary~1.4}.
For any stationary boundary, Lemma~\ref{lem:stationarity-cmc} and
the same index decomposition give \textup{(ii)} $\Leftrightarrow$
\textup{(iii)}. Together with \textup{(ii)} $\Rightarrow$ \textup{(i)}
proved in Section~\ref{rig:sec:classification}, this completes the theorem.
\end{proof}

\begin{proof}[Proof of \cref{rig:cor:isoperimetric}]
Leonardi--Rigot~\cite{LeonardiRigot} give a bounded
open representative $E_0$ of $E$. Choose a Euclidean closed ball
containing $E_0$. The set $\Omega$ is contained in that ball:
otherwise its openness would give a ball of positive measure in
$\Omega\setminus E_0$, contradicting $|\Omega\triangle E_0|=0$.
Thus $\Omega$ is bounded and has compact $C^2$ boundary.
Minimality among sets of the same volume makes every exactly
volume-preserving ambient variation stationary with nonnegative
second perimeter derivative. The boundary is therefore
volume-preserving stable. Theorem~\ref{rig:thm:stable-cmc} identifies
it as a Pansu sphere; its bounded side is the corresponding Pansu ball.
\end{proof}

\begin{applicationremark}[The case $n=1$]
The assumption $n\ge2$ enters both the geometric and analytic parts
of the proof. In Lemma~\ref{lem:horizontal-constancy}, it ensures
that $d\theta$ on the horizontal tangent space has positive rank
$2n-2$, which forces a horizontally constant function to be constant.
Near the characteristic set, the cofactor rank bound
$\operatorname{rank}D_zG\ge n$ gives local integrability of $g^{-1}$,
where $g=|G|$, and a cutoff flux error $O(\rho^{n-1})\to0$;
these justify differentiation
\cite{JacobiSpectrum}*{Corollary~2.2} and divergence extension
by Lemma~\ref{app:lem:no-divergence-charge}.
For $n=1$, the horizontal rank $2n-2$ vanishes, and the cofactor
rank bound alone no longer yields these analytic conclusions.
Moreover, for $n=1$ the dilation speed $1+u\tan u$ satisfies
$|\partial_u(1+u\tan u)|\asymp r^{-2}$ as $r=\cos u\downarrow0$.
Since $d\sigma_H\asymp r^2\,dr\,d\omega$, its energy contains a
divergent integral $\int_0^\varepsilon r^{-2}\,dr$. Thus the global
energy representation used above is unavailable.
\end{applicationremark}

\appendix
\section{A divergence identity for the shape operator}
\label{app:codazzi-graph}
The centre calculation uses Codazzi only through the divergence of
$\ShapeError e_n$. We compute on smooth regular patches, then pass
to $C^2$.
Contract the symmetrized Codazzi identity of
\cite{GPV}*{Corollary~3.15}, which follows from its Proposition~3.14.
Let
$\mathcal D_XY=(\nabla_XY)^{\top_H}$ be the induced horizontal
Tanaka--Webster connection. The source's outward horizontal normal
is $-e_{2n}$, its $J$-rotation is $e_n$, and its $\tilde h$ is our $\Shape$.
For a horizontal tangent orthonormal frame $(E_a)_{a=1}^{2n-1}$,
substitute $X=e_n$, $Y=E_a$ in that corollary and sum.

Since
$\langle JE_a,e_n\rangle=-\langle JE_a,Je_{2n}\rangle
=-\langle E_a,e_{2n}\rangle=0$ and
$\operatorname{tr}(\mathcal D_{e_n}\Shape)=e_nH_\Sigma$, this gives
\[
 \begin{aligned}
 \sum_a\langle(\mathcal D_{E_a}\Shape)e_n,E_a\rangle
 &=\sum_a\bigl\langle\mathcal D_{E_a}(\Shape e_n)
   -\Shape\mathcal D_{E_a}e_n,E_a\bigr\rangle\\
 &=e_nH_\Sigma+\alpha\sum_a\bigl(
   \langle\Shape E_a,E_a\rangle
   -\langle e_n,E_a\rangle\langle\Shape E_a,e_n\rangle\bigr)\\
 &=e_nH_\Sigma+\alpha H_\Sigma-\alpha\langle\Shape e_n,e_n\rangle.
 \end{aligned}
\]
By \eqref{rig:eq:characteristic-direction-derivative}, the
self-adjointness of $\Shape$ and the skew symmetry of $J'$, the
remaining product-rule term is
\[
 \begin{aligned}
 \sum_a\langle\Shape\mathcal D_{E_a}e_n,E_a\rangle
 &=\sum_a\bigl\langle\Shape\bigl[J'\Shape E_a
   +\alpha(E_a-\langle E_a,e_n\rangle e_n)\bigr],E_a\bigr\rangle\\
 &=\operatorname{tr}(\Shape J'\Shape)
   +\alpha(H_\Sigma-\langle\Shape e_n,e_n\rangle)\\
 &=\sum_a\langle J'\Shape E_a,\Shape E_a\rangle
   +\alpha(H_\Sigma-\langle\Shape e_n,e_n\rangle)\\
 &=\alpha(H_\Sigma-\langle\Shape e_n,e_n\rangle).
 \end{aligned}
\]
Adding the weighted correction $2\alpha \langle\Shape e_n,e_n\rangle$ from
\eqref{eq:weighted-horizontal-divergence} proves
\[
 \begin{aligned}
 \operatorname{div}_H(\Shape e_n)
 &=\sum_a\langle\mathcal D_{E_a}(\Shape e_n),E_a\rangle
   +2\alpha\langle\Shape e_n,e_n\rangle\\
 &=e_nH_\Sigma+2\alpha(H_\Sigma-\langle\Shape e_n,e_n\rangle)
   +2\alpha\langle\Shape e_n,e_n\rangle\\
 &=e_nH_\Sigma+2\alpha H_\Sigma.
 \end{aligned}
\]
For a smooth patch with $H_\Sigma=2n$, subtracting
$Pe_n=2e_n$ and using
$\operatorname{div}_He_n=(2n-2)\alpha+2\alpha=2n\alpha$ gives
\[
 \operatorname{div}_H(\ShapeError e_n)
 =0+2\alpha(2n)-2(2n\alpha)=0.
\]

\emph{Weak extension on regular patches.}
Let $\iota:U\subset\RR^{2n}\to\Sigma^*$ be a noncharacteristic
$C^2$ parametrization with $H_\Sigma=2n$. Fix
$\psi\in C_c^1(U)$ and a compact subpatch $K\Subset U$ containing
its support. Approximate $\iota$ in $C^2$ near $K$ by smooth
$\iota_j$, orienting them so that their unit normals converge uniformly
to the unit normal of $\iota$ on $K$. Noncharacteristicity
persists on $K$ for large $j$, but the approximants need not be CMC.
Put $\tau_j=H_j-2n$ and
$\mathsf E_{P,j}=\Shape_j-P_j$. Since $P_je_{n,j}=2e_{n,j}$,
the smooth identities above give
\[
 \begin{aligned}
 \operatorname{div}_{H,j}(\mathsf E_{P,j}e_{n,j})
 &=\operatorname{div}_{H,j}(\Shape_je_{n,j})
   -2\operatorname{div}_{H,j}e_{n,j}\\
 &=e_{n,j}H_j+2\alpha_jH_j-4n\alpha_j\\
 &=e_{n,j}\tau_j+2\alpha_j\tau_j,\\
 \operatorname{div}_{H,j}(\tau_je_{n,j})
 &=e_{n,j}\tau_j+\tau_j\operatorname{div}_{H,j}e_{n,j}\\
 &=e_{n,j}\tau_j+2n\alpha_j\tau_j.
 \end{aligned}
\]
Thus the corrected horizontal tangent field
$Z_j=(\mathsf E_{P,j}-\tau_j\Id)e_{n,j}$ satisfies
\[
 \begin{aligned}
 \operatorname{div}_{H,j}Z_j
 &=\operatorname{div}_{H,j}(\mathsf E_{P,j}e_{n,j})
   -\operatorname{div}_{H,j}(\tau_je_{n,j})\\
 &=(e_{n,j}\tau_j+2\alpha_j\tau_j)
   -(e_{n,j}\tau_j+2n\alpha_j\tau_j)\\
 &=-(2n-2)\alpha_j\tau_j.
 \end{aligned}
\]
The derivatives $e_{n,j}\tau_j=e_{n,j}H_j$ have cancelled, so no
convergence of curvature derivatives is needed.

Write $\rho_j\,dx=\iota_j^*d\sigma_{H,j}$ and let
$\widehat Z_j=(d\iota_j)^{-1}(Z_j\circ\iota_j)$ be the coordinate
vector field on $U$. In the following integrals, scalar coefficients
are also pulled back by $\iota_j$. The coordinate formula for
weighted divergence and integration by parts give
\[
 \begin{aligned}
 \operatorname{div}_x(\rho_j\widehat Z_j)
 &=\rho_j(\operatorname{div}_{H,j}Z_j)\circ\iota_j
   =-(2n-2)\rho_j\alpha_j\tau_j,\\
 \int_U\rho_j\widehat Z_j\cdot D\psi\,dx
 &=-\int_U\psi\operatorname{div}_x(\rho_j\widehat Z_j)\,dx\\
 &=(2n-2)\int_U\rho_j\alpha_j\tau_j\psi\,dx.
 \end{aligned}
\]
There is no boundary term because $\psi$ is compactly supported.
On $K$, all these coefficients depend continuously on the first
and second derivatives of the parametrization; the horizontal normal
length stays bounded away from zero. Hence $C^2$ convergence gives,
uniformly on $K$,
\[
 \rho_j\to\rho>0,\qquad \alpha_j\to\alpha,\qquad
 \tau_j=H_j-2n\to0,\qquad
 \widehat Z_j\to\widehat Z
 :=(d\iota)^{-1}\bigl((\ShapeError e_n)\circ\iota\bigr).
\]
In particular, both limiting steps are controlled by
\[
 \begin{aligned}
 \left|(2n-2)\int_U\rho_j\alpha_j\tau_j\psi\,dx\right|
 &\le(2n-2)\|\rho_j\alpha_j\|_{L^\infty(K)}
   \|\tau_j\|_{L^\infty(K)}\|\psi\|_{L^1(K)}\longrightarrow0,\\
 \left|\int_U(\rho_j\widehat Z_j-\rho\widehat Z)
   \cdot D\psi\,dx\right|
 &\le\|\rho_j\widehat Z_j-\rho\widehat Z\|_{L^\infty(K)}
   \|D\psi\|_{L^1(K)}\longrightarrow0.
 \end{aligned}
\]
For $\varphi=\psi\circ\iota^{-1}$, we obtain
\eqref{rig:eq:codazzi-conservation} weakly, also across $\alpha=0$
since no step divides by $\alpha$:
\[
 \int_{\iota(U)}\langle\ShapeError e_n,\nabla_b\varphi\rangle
   \,d\sigma_H
 =\int_U\rho\widehat Z\cdot D\psi\,dx=0.
\]

\section{Divergence extension across the characteristic set}
\label{app:divergence-extension}
In a characteristic graph chart $t=f(z)$, put $G=\nabla f+Jz$ and
$g=|G|$, so that $d\sigma_H=g\,dz$. The cofactor rank and sublevel
estimates are proved in \cite{JacobiSpectrum}*{Corollary~2.2}. They make
the cutoff flux vanish for horizontal fields with $O(g^{-1})$ growth,
even when the characteristic set has positive dimension.

\begin{samepage}
  \begin{lemma}
    \label{app:lem:no-divergence-charge}
    Let $n\ge2$, and let $\Sigma$ be compact, oriented and $C^2$.
    In characteristic graph charts put $g=|G|$ as above. Suppose a locally
    integrable horizontal tangent field $Z$ on $\Sigma^*$ satisfies
    $|Z|\le Cg^{-1}$ and $\operatorname{div}_HZ=f$ weakly there, with
    $f\in L^1(\Sigma,d\sigma_H)$.
    Then the identity extends distributionally across $\Sigma\setminus\Sigma^*$.
    In particular, if $\partial\Sigma=\varnothing$, then
    \[
      \int_\Sigma f\,d\sigma_H=0.
    \]
  \end{lemma}
\end{samepage}

\begin{proof}
  Fix a compactly supported $C^1$ test $\psi$ in a characteristic
  graph chart and put $K=\operatorname{supp}\psi$.
  In these coordinates $d\sigma_H=g\,dz$. Set $Z=0$ on $\{g=0\}$;
  this does not affect any integral with respect to $d\sigma_H$.
  The growth assumption gives
  \[
    \int_K|Z|\,d\sigma_H
    =\int_K|Z|g\,dz\le C|K|<\infty.
  \]
  For $E_\rho=K\cap\{g<2\rho\}$,
  the sublevel estimate in \cite{JacobiSpectrum}*{Corollary~2.2}
  with $\gamma=0$ gives
  \[
    |E_\rho|
    =\int_{K\cap\{g<2\rho\}}1\,dz
    \le C(2\rho)^n\le C\rho^n.
  \]

  Choose $\theta\in C^\infty(\RR)$ with $0\le\theta\le1$,
  $\theta=0$ on $(-\infty,1]$ and $\theta=1$ on $[2,\infty)$,
  and set $\chi_\rho(z)=\theta(g(z)/\rho)$.
  On $\{g>0\}$ the chain rule gives
  \[
    \partial_{z_a}g
    =\frac{\langle G,\partial_{z_a}G\rangle}{|G|},
    \qquad
    D_z\chi_\rho
    =\rho^{-1}\theta'(g/\rho)D_zg,
    \qquad
    |D_z\chi_\rho|
    \le C\rho^{-1}\mathbf1_{\{\rho<g<2\rho\}}.
  \]
  Here $D_zG$ is bounded on the fixed compact chart. The function
  $\chi_\rho$ is identically zero near $\{g=0\}$, so
  $\psi\chi_\rho$ is a $C^1$ test with compact support in $\Sigma^*$.
  The weak divergence identity applies to this test, by $C^1$
  approximation on its compact regular support if needed. Thus
  \[
    \begin{aligned}
      \int\psi\chi_\rho f\,d\sigma_H
      &=-\int\langle Z,\nabla_b(\psi\chi_\rho)\rangle\,d\sigma_H\\
      &=-\int\chi_\rho\langle Z,\nabla_b\psi\rangle\,d\sigma_H
        -\int\psi\langle Z,\nabla_b\chi_\rho\rangle\,d\sigma_H.
    \end{aligned}
  \]
  Using $1=\chi_\rho+(1-\chi_\rho)$ gives
  \[
    \begin{aligned}
      \int\psi f\,d\sigma_H
        +\int\langle Z,\nabla_b\psi\rangle\,d\sigma_H
      &=\int\psi(1-\chi_\rho)f\,d\sigma_H\\
      &\quad+\int(1-\chi_\rho)\langle Z,\nabla_b\psi\rangle\,d\sigma_H\\
      &\quad-\int\psi\langle Z,\nabla_b\chi_\rho\rangle\,d\sigma_H.
    \end{aligned}
  \]
  Since $Z$ is horizontal and tangent to the graph, its $z$-coordinate
  vector has Euclidean length $|Z|$. Therefore, for any $C^1$ function $u$,
  \[
    |\langle Z,\nabla_bu\rangle|=|Zu|\le|Z|\,|D_zu|.
  \]
  Using this inequality, $d\sigma_H=g\,dz$ and $|Z|g\le C$, the
  cutoff derivative term satisfies
  \[
    \begin{aligned}
      \left|\int\psi\langle Z,\nabla_b\chi_\rho\rangle\,d\sigma_H\right|
      &\le\|\psi\|_\infty
          \int_{K\cap\{\rho<g<2\rho\}}|Z|\,|D_z\chi_\rho|\,g\,dz\\
      &\le C\|\psi\|_\infty\rho^{-1}
          \int_{K\cap\{\rho<g<2\rho\}}|Z|g\,dz\\
      &\le C\|\psi\|_\infty\rho^{-1}|E_\rho|
       \le C\|\psi\|_\infty\rho^{n-1}\longrightarrow0.
    \end{aligned}
  \]
  The last limit uses $n\ge2$. The omitted field term satisfies
  \[
    \begin{aligned}
      \left|\int(1-\chi_\rho)\langle Z,\nabla_b\psi\rangle\,d\sigma_H\right|
      &\le\|D_z\psi\|_\infty\int_{E_\rho}|Z|g\,dz\\
      &\le C\|D_z\psi\|_\infty|E_\rho|
       \le C\|D_z\psi\|_\infty\rho^n\longrightarrow0.
    \end{aligned}
  \]
  Finally,
  \[
    \left|\int\psi(1-\chi_\rho)f\,d\sigma_H\right|
    \le\|\psi\|_\infty\int_{E_\rho}|f|\,d\sigma_H
    \longrightarrow0.
  \]
  Indeed, $\mathbf1_{E_\rho}\to0$ almost everywhere with respect to
  $d\sigma_H=g\,dz$, and $|f|\in L^1(d\sigma_H)$ supplies an integrable
  dominating function. Letting $\rho\downarrow0$ in the three-term
  identity proves
  \[
    \int\psi f\,d\sigma_H
    =-\int\langle Z,\nabla_b\psi\rangle\,d\sigma_H
  \]
  across the characteristic set.

  For a general test $\psi$ supported in the interior of $\Sigma$,
  choose a finite partition $\sum_\ell\eta_\ell=1$ near its support,
  subordinate to regular and characteristic charts.
  Apply the local identity to $\eta_\ell\psi$ and sum:
  \[
    \begin{aligned}
      \int_\Sigma\psi f\,d\sigma_H
      &=\sum_\ell\int_\Sigma\eta_\ell\psi f\,d\sigma_H\\
      &=-\sum_\ell\int_\Sigma
          \langle Z,\nabla_b(\eta_\ell\psi)\rangle\,d\sigma_H\\
      &=-\int_\Sigma\left(\sum_\ell\eta_\ell\right)
          \langle Z,\nabla_b\psi\rangle\,d\sigma_H
        -\int_\Sigma\psi\left\langle Z,
          \nabla_b\sum_\ell\eta_\ell\right\rangle\,d\sigma_H\\
      &=-\int_\Sigma\langle Z,\nabla_b\psi\rangle\,d\sigma_H.
    \end{aligned}
  \]
  If $\partial\Sigma=\varnothing$, compactness permits $\psi\equiv1$,
  giving
  \[
    \int_\Sigma f\,d\sigma_H
    =-\int_\Sigma\langle Z,\nabla_b1\rangle\,d\sigma_H=0.
  \]
\end{proof}

\begin{applicationremark}[Growth threshold]
  More generally, suppose the cofactor rank is at least $r$ at every
  characteristic point of a compact graph patch. The proof of
  \cite{JacobiSpectrum}*{Lemma~2.1}, applied with rank $r$, gives
  $|\{g<\rho\}|\le C\rho^r$. If $|Z|\le Cg^{-s}$ with
  $0\le s<r$, the cutoff error is bounded by
  \[
    C\rho^{-1}\int_{\{\rho<g<2\rho\}}g^{1-s}\,dz
    \le C\rho^{r-s}\longrightarrow0.
  \]
  The weighted sublevel estimate from the same proof makes the omitted
  field integral $O(\rho^{r+1-s})$.
  Thus the same extension holds,
  provided the weak divergence still belongs to $L^1(d\sigma_H)$.
  At the endpoint $s=r$, suppose
  $\sup_{\{0<g<\rho\}}g^r|Z|\to0$ as $\rho\downarrow0$ on the fixed
  compact graph patch. Then the cutoff flux also vanishes; an
  $O(g^{-r})$ bound alone does not force this limit. The lemma uses
  $r\ge n$ and $s=1$, so $n\ge2$ places its fields strictly below
  the growth threshold.
\end{applicationremark}


\begin{thebibliography}{99}
\linespread{1.0}\selectfont
\setlength{\itemsep}{3pt}
\setlength{\parsep}{0pt}
\bibitem{BarbosaDoCarmo}
J.~L. Barbosa and M.~do Carmo,
\href{https://doi.org/10.1007/BF01215045}{\emph{Stability of hypersurfaces with constant mean curvature}},
Math. Z. \textbf{185} (1984), 339--353.

\bibitem{BFM}
T.~P. Branson, L.~Fontana, and C.~Morpurgo,
\href{https://doi.org/10.4007/annals.2013.177.1.1}{\emph{Moser--Trudinger and Beckner--Onofri's inequalities on the CR sphere}},
Ann. of Math. \textbf{177} (2013), no.~1, 1--52.

\bibitem{CDPT}
L. Capogna, D. Danielli, S. D. Pauls, and J. T. Tyson,
\href{https://doi.org/10.1007/978-3-7643-8133-2}{\emph{An Introduction to the Heisenberg Group and the Sub-Riemannian
Isoperimetric Problem}},
Progr. Math., vol.~259, Birkh\"auser, Basel, 2007.

\bibitem{JacobiSpectrum}
W.-C. Chen, W.-T. Kao and C.-K. Lien,
\emph{The Jacobi Spectrum and Global Second Variation of Pansu Spheres},
preprint, 2026, submitted concurrently to arXiv.

\bibitem{CCHYSigma}
J.-H. Cheng, H.-L. Chiu, J.-F. Hwang, and P. Yang,
\href{https://doi.org/10.1007/s00526-016-1006-7}{\emph{Umbilic hypersurfaces of constant sigma-$k$ curvature in the Heisenberg group}},
Calc. Var. Partial Differential Equations \textbf{55} (2016), no.~3, article~66.

\bibitem{CCHY}
J.-H. Cheng, H.-L. Chiu, J.-F. Hwang, and P. Yang,
\href{https://doi.org/10.1515/crelle-2015-0044}{\emph{Umbilicity and characterization of Pansu spheres in the Heisenberg group}},
J. Reine Angew. Math. \textbf{738} (2018), 203--235.

\bibitem{CHY}
J.-H.~Cheng, J.-F.~Hwang, and P.~Yang,
\href{https://doi.org/10.1007/s00208-006-0033-7}{\emph{Existence and uniqueness for $p$-area minimizers in the
Heisenberg group}}, Math. Ann. \textbf{337} (2007), 253--293.

\bibitem{CLL}
H.-L. Chiu, S.-H. Lai, and H.-F. Liu,
\href{https://doi.org/10.1007/s12220-026-02594-8}{\emph{A complete classification of rotationally symmetric hypersurfaces
in the Heisenberg groups $H_n$}},
J. Geom. Anal. \textbf{36} (2026), article~348.

\bibitem{DGN}
D.~Danielli, N.~Garofalo, and D.-M.~Nhieu,
\href{https://doi.org/10.1016/j.aim.2007.04.004}{\emph{Sub-Riemannian calculus on hypersurfaces in Carnot groups}},
Adv. Math. \textbf{215} (2007), 292--378.

\bibitem{GPV}
G.~Giovannardi, A.~Pinamonti, and S.~Verzellesi,
\href{https://arxiv.org/abs/2409.20359v2}{\emph{Curvature estimates for minimal hypersurfaces in the Heisenberg group}},
\href{https://arxiv.org/abs/2409.20359v2}{arXiv:2409.20359v2 [math.DG]}, 28 May 2025.

\bibitem{HPPreprint}
R.~K. Hladky and S.~D. Pauls,
\href{https://arxiv.org/abs/math/0702237v1}{\emph{Variation of perimeter measure in sub-Riemannian geometry}},
\href{https://arxiv.org/abs/math/0702237v1}{arXiv:\allowbreak math/0702237v1 [math.DG]}, 8 February 2007.

\bibitem{HP}
R.~K. Hladky and S.~D. Pauls,
\href{https://dergipark.org.tr/en/pub/iejg/article/597626}{\emph{Variation of perimeter measure in sub-Riemannian geometry}},
Int. Electron. J. Geom. \textbf{6} (2013), no.~1, 8--40.

\bibitem{HurtadoRosalesSpheres}
A.~Hurtado and C.~Rosales,
\href{https://doi.org/10.1007/s00526-015-0898-y}{\emph{Existence, characterization and stability of Pansu spheres in
sub-Riemannian 3-space forms}},
Calc. Var. Partial Differential Equations \textbf{54} (2015), no.~3,
3183--3227.

\bibitem{HurtadoRosalesStable}
A.~Hurtado and C.~Rosales,
\href{https://doi.org/10.1007/s00526-020-01834-1}{\emph{An instability criterion for volume-preserving area-stationary
surfaces with singular curves in sub-Riemannian 3-space forms}},
Calc. Var. Partial Differential Equations \textbf{59} (2020), article~165.

\bibitem{LeonardiRigot}
G.~P. Leonardi and S.~Rigot,
\href{https://iris.unitn.it/handle/11572/229177}{\emph{Isoperimetric sets on Carnot groups}},
Houston J. Math. \textbf{29} (2003), no.~3, 609--637.

\bibitem{MontefalconeVar}
F.~Montefalcone,
\href{https://doi.org/10.1016/j.matpur.2007.01.009}{\emph{Hypersurfaces and variational formulas in sub-Riemannian Carnot
groups}}, J. Math. Pures Appl. \textbf{87} (2007), 453--494.

\bibitem{MontefalconeStability}
F. Montefalcone,
\href{https://arxiv.org/abs/1110.0707v2}{\emph{Stability of Heisenberg isoperimetric profiles}},
\href{https://arxiv.org/abs/1110.0707v2}{arXiv:1110.0707v2 [math.MG]}, 17 November 2011.

\bibitem{Pansu}
P. Pansu,
\href{https://www.imo.universite-paris-saclay.fr/~pierre.pansu/CRAS_1982.pdf}{\emph{Une in\'egalit\'e isop\'erim\'etrique sur le groupe de Heisenberg}},
C. R. Acad. Sci. Paris S\'er. I Math. \textbf{295} (1982), 127--130.

\bibitem{RitoreCalibration}
M. Ritor\'e,
\href{https://doi.org/10.1007/s00526-011-0425-8}{\emph{A proof by calibration of an isoperimetric inequality in the
Heisenberg group $\mathbb H^n$}},
Calc. Var. Partial Differential Equations \textbf{44} (2012), 47--60.

\bibitem{RitoreRosales}
M. Ritor\'e and C. Rosales,
\href{https://doi.org/10.1007/bf02922137}{\emph{Rotationally invariant hypersurfaces with constant mean curvature in
the Heisenberg group $\mathbb H^n$}},
J. Geom. Anal. \textbf{16} (2006), no.~4, 703--720.

\bibitem{RitoreRosalesStationary}
M.~Ritor\'e and C.~Rosales,
\href{https://doi.org/10.1016/j.aim.2008.05.011}{\emph{Area-stationary surfaces in the Heisenberg group $\mathbb H^1$}},
Adv. Math. \textbf{219} (2008), no.~2, 633--671.

\bibitem{Wente}
H.~C. Wente,
\href{https://msp.org/pjm/1991/147-2/pjm-v147-n2-p13-p.pdf}{\emph{A note on the stability theorem of J.~L. Barbosa and M.~do Carmo for closed surfaces of constant mean curvature}},
Pacific J. Math. \textbf{147} (1991), no.~2, 375--379.

\end{thebibliography}
\end{document}